\documentclass[reqno]{amsart}

\usepackage{amsmath,amsthm, txfonts, graphicx}

\newtheorem{theorem}{Theorem}[section]
\newtheorem{lemma}[theorem]{Lemma}

\newtheorem{corollary}[theorem]{Corollary}
\newtheorem{conjecture}[theorem]{Conjecture}
\numberwithin{equation}{section}

\theoremstyle{definition}
\newtheorem{definition}[theorem]{Defintion}
\newtheorem{remark}[theorem]{Remark}
\newtheorem{example}[theorem]{Example}

\DeclareMathOperator{\dom}{dom}%
\DeclareMathOperator{\sppt}{Sppt}%
\DeclareMathOperator{\singsppt}{SingSppt}%
\DeclareMathOperator{\interior}{int}%

\newcommand{\DF}{\ensuremath{E}}
\newcommand{\domain}{\Omega}
\newcommand{\testfns}{\ensuremath{\mathcal{D}(\Omega)}}
\newcommand{\distribs}{\ensuremath{\mathcal{D'}(\Omega)}}
\newcommand{\smoothfns}{\ensuremath{\mathcal{E}(\Omega)}}
\newcommand{\compactdistribs}{\ensuremath{\mathcal{E'}(\Omega)}}
\newcommand{\testfn}{\phi}
\newcommand{\distrib}{T}
\newcommand{\vstr}[1][3]{\rule{0ex}{#1ex}}
\newcommand{\negsp}[1][20]{\mspace{-#1mu}}
\newcommand{\evald}[2][]{\ensuremath{\negsp[4]\left.\vstr[2.0] \right|_{#2}^{#1}}} 
\newcommand{\distribson}{\ensuremath{\mathcal{D}'}} 
\newcommand{\testfnson}{\ensuremath{\mathcal{D}}}
\newcommand{\smoothfnson}{\ensuremath{\mathcal{E}}}

\begin{document}

\title{Distribution Theory on P.C.F. Fractals}
\author{Luke G. Rogers}
\address{University of Connecticut, Storrs, CT 06269-3009 USA}
\email{rogers@math.uconn.edu}

\author{Robert S. Strichartz}
\address{Cornell University, Ithaca, NY 14850-4201 USA}
\email{str@math.cornell.edu}
\thanks{Second author supported in part by NSF grant DMS 062440.}

\date{\today}

\begin{abstract}
We construct a theory of distributions in the setting of analysis on post-critically finite
self-similar fractals, and on fractafolds and products based on such fractals.   The results
include basic properties of test functions and distributions, a structure theorem showing that
distributions are locally-finite sums of powers of the Laplacian applied to continuous functions,
and an analysis of the distributions with point support.  Possible future applications to the study
of hypoelliptic partial differential operators are suggested.
\end{abstract}

\maketitle


\section{Introduction}

The prevalence of fractal-like objects in nature has led both physicists and mathematicians to
study dynamic processes on fractals.  One rigorous way to do this on post-critically finite
(p.c.f.) fractals is by studying differential equations in the natural analytic structure. A brief
description of this analytic structure will appear in Section \ref{setting_section} below, but we
emphasize that it is intrinsic to the fractal, and is not necessarily related to the analysis on a
space in which the fractal may be embedded. For example, the familiar Sierpinski gasket fractal SG
is often visualized as a subset of $\mathbb{R}^{2}$, but restricting a smooth function on
$\mathbb{R}^{2}$ to SG does not give a smooth function on the fractal \cite{MR1707752}. Similarly,
we should not expect the solutions of differential equations on fractals to be quite like the
solutions of their Euclidean analogues; for example, many fractals have Laplacian eigenfunctions
that vanish identically on large open sets \cite{MR1489140}, whereas eigenfunctions of the
Euclidean Laplacian are analytic.

Perhaps the most important tools for studying differential equations in the Euclidean context are
Fourier analysis and the theory of distributions.  Since the theory of analysis on fractals relies
on first constructing a Laplacian operator $\Delta$, it is unsurprising that quite a lot is known
about the fractal analogue of Fourier analysis. In some interesting cases the spectrum and
eigenfunctions of the Laplacian are known explicitly, and many results about Laplacian
eigenfunctions have also been derived by using probability theory to study the heat diffusion on
fractals. Fourier-type techniques have also been applied to treat smoothness in the fractal
setting: analogues of the Sobolev, H\"{o}lder-Zygmund and Besov spaces that are so important in
Euclidean analysis of differential equations were introduced and investigated in \cite{MR1962353}.
Analogues of other basic objects in Euclidean analysis are studied in~\cite{MR2150975,CalculusII}.
By contrast there has not previously been a theory of distributions on fractals, and it is the
purpose of the present work to provide one.

It is relatively elementary to define distributions on fractals; as usual they are dual to the
space of smooth functions with compact support, where a function $u$ is said to be smooth if
$\Delta^{k}u$ is continuous for all $k\in\mathbb{N}$.  The main theorems about distributions are
then really theorems about smooth functions, and the key to proving many of them is knowing how to
smoothly partition smooth functions. Partitions of unity are used to achieve this in the Euclidean
setting, but are not useful on fractals because products of smooth functions are not
smooth~\cite{MR1707752}. (This latter fact also implies that products of functions and
distributions are not distributions, so the distributions are not a module over the smooth
functions.) Instead we rely on a partitioning theorem for smooth functions proved in~\cite{RST},
see Theorem~\ref{setting_partitionthm} below. Using this partitioning result we are able to prove
analogues of the standard structure theorems describing distributions as derivatives of continuous
functions (Theorem \ref{structure_structuretheorem}), and identifying the positive distributions as
positive measures (Theorem \ref{structure_positivedistribsarepositivemeasures}). We can also
characterize the distributions of point support as finite linear combinations of certain
``derivatives'' of Dirac masses that can be explicitly described
(Corollaries~\ref{distribsatapoint_identifyingdistribsatjunctionpoint}
and~\ref{distribsatapoint_identifydistribsatgenericpoints}), provided we make certain assumptions
about the point in question.  These assumptions are needed in order to understand the local
behavior of smooth functions at the point, and are related to work done
in~\cite{MR1761364,MR1761365,MR2219003,pelanderteplyaev07,MR2417415}. The reader should be warned
that many of our proofs are quite technical in nature; we have tried to explain in advance the
strategies behind the proofs, which are more conceptual.

At the end of this paper we suggest several interesting questions related to the hypoellipticity of
differential operators that are natural to consider in the context of distribution theory. It
should also be noted that there are a number of results on local solvability of differential
equations~\cite{MR2177187,MR2346562} that could be reformulated in this context. We expect that
this work will provide the foundation for many subsequent investigations.


\section{Setting}\label{setting_section}

We begin by describing the basic elements of analysis on a post-critically finite self-similar set
$X$, as laid out in the monograph of Kigami~\cite{Kigamibook}; in this section all unreferenced
results may be found in~\cite{Kigamibook}, which also includes proofs and references to the
original literature. The reader who prefers to have a concrete example of a p.c.f. set in mind may
choose to think of $X$ as the Sierpinski Gasket, in which case an more elementary exposition of the
material that follows may be found in \cite{Strichartzbook}.

\subsection*{P.C.F. Fractals}


Let $X$ be a self-similar subset of $\mathbb{R}^{d}$ (or more generally a compact metric space). By
this we mean that there are contractive similarities $\{F_{j}\}_{j=1}^{N}$ of $\mathbb{R}^{d}$, and
$X$ is the unique compact set satisfying $X=\cup_{j=1}^{N}F_{j}(X)$.  Then $X$ has a natural cell
structure in which we associate to a word $w=w_{1}w_{2}\dotsc w_{m}$ of length $m$ the map
$F_{w}=F_{w_{1}}\circ\dotsm\circ F_{w_{m}}$, and call $F_{w}(X)$ an $m$-cell.  If $w$ is an
infinite word then we let $[w]_{m}$ be its length $m$ truncation and note that
$F_{w}(X)=\bigcap_{m}F_{[w]_{m}}(X)$ is a point in $X$.

We say $F_{j}(x)$ is a critical value of $X=\cup_{l=1}^{N}F_{l}(X)$ if there is $y\in X$ and $k\neq
j$ such that $F_{j}(x)=F_{k}(y)$. An infinite word $w$ is critical if $F_{w}(X)$ is a critical
value, and $\tilde{w}$ is post-critical if there is $j\in\{1,\dotsc,N\}$ such that $jw$ is
critical. We always assume that the set of post-critical words is finite, in which case the fractal
is said to be {\em post-critically finite (p.c.f.)}. The {\em boundary} of $X$ is then defined to
be the finite set $V_0$ consisting of all points $F_{\tilde{w}}(X)$ for which $\tilde{w}$ is
post-critical; this set is assumed to contain at least two points.  We also let
$V_{m}=\cup_{w}F_{w}(V_{0})$, where the union is over all words of length $m$.  Points in
$V_\ast=\cup_{m\geq 0}V_{m}$ that are not in $V_{0}$ are called {\em junction points}, and a key
property of p.c.f. fractals is that cells intersect only at junction points.

We fix a probability measure $\mu$ on $X$ that is self-similar in the sense that there are
$\mu_{1},\dotsc,\mu_{N}$ such that the cell corresponding to $w=w_{1}\dotsc w_{m}$ has measure
$\mu(F_{w}(X))=\prod_{j=1}^{m}\mu_{w_{j}}$.  The usual Bernoulli measure in which each
$\mu_{j}=\frac{1}{N}$ is one example.

\subsection*{Dirichlet Form}
Our analysis on $X$ will be constructed from a self-similar Dirichlet form. A closed quadratic form
\DF\ on $L^{2}(\mu)$ is called Dirichlet if it has the (Markov) property that if $u\in\dom(\DF)$
then so is $\tilde{u}=u\chi_{0<u<1} +\chi_{u\geq1}$ and $\DF(\tilde{u},\tilde{u})\leq\DF(u,u)$,
where $\chi_{A}$ is the characteristic function of $A$. Self-similarity of \DF\ means that there
are renormalization factors $r_{1},\dotsc,r_{N}$ such that
\begin{equation}\label{Setting-DFisselfsimilar}
    \DF(u,v) = \sum_{j=1}^{N} r_{j}^{-1} \DF(u\circ F_{j},v\circ F_{j}).
    \end{equation}
It follows immediately that $\DF(u,v)$ can also be expressed as the sum over $m$-words of
$r_{w}^{-1} \DF(u\circ F_{w},v\circ F_{w})$ where $r_{w}=r_{w_{1}}\dotsm r_{w_{m}}$. In order to
use results from \cite{RST} we assume that $0<r_{j}<1$ for all $j$, in which case $\DF$ is {\em
regular}, meaning that that $C(X)\bigcap\dom(\DF)$ is dense both in $\dom(\DF)$ with \DF-norm and
in the space of continuous functions $C(X)$ with supremum norm.

It is far from obvious that interesting fractals should support such Dirichlet forms, but in fact
the conditions described so far are satisfied by many p.c.f. self-similar sets that have sufficient
symmetry. In particular, if $X$ is a nested fractal in the sense of Lindstr\o m \cite{MR988082}
then a Dirichlet form of the above type may be constructed using a diffusion or a harmonic
structure \cite{MR1442498,MR1301625,MR1474807}. Some other approaches may be found in
\cite{MR1769995,MR2017320,MR2105771,MR2359927,MR2267882,MR2254554,PeironePreprint}.

\subsection*{Harmonic Functions}
Given a function on $V_{0}$ (usually thought of as an assignment of boundary values) there is a
unique continuous function on $X$ that has these boundary values and minimizes the energy.  Such
functions are called {\em harmonic}, and form a finite dimensional space containing the constants.
It is easy to see that there are {\em harmonic extension matrices} $A_{j}$, $j=1,\dotsc,N$ with the
property that if $h$ is harmonic then $A_{j}$ maps the values of $h$ on $V_{0}$ to its values on
$F_{j}(V_{0})$. The largest eigenvalue of each $A_{j}$ is $1$, corresponding to the constant
functions; it is useful to know that the second eigenvalue is $r_{j}$, and that all other
eigenvalues (which may be complex) have strictly smaller absolute value (\cite{Kigamibook},
Appendix~A).

\subsection*{The Laplacian and Normal Derivatives}
Using the energy and measure we produce a weak Laplacian by defining $f=\Delta u$ if $\DF(u,v)=
-\int fv\,d\mu$ for all $v\in\dom(\DF)$ that vanish on $V_0$.  Our assumptions so far are
sufficient to conclude that $-\Delta$ is a non-negative self-adjoint operator on $L^{2}(\mu)$ with
compact resolvent (see Theorem 2.4.2 of \cite{Kigamibook}).  We denote its eigenvalues by
$\lambda_j$ and the corresponding eigenvectors by $\psi_j$.  When $\Delta u\in C(X)$ we write
$u\in\dom(\Delta)$ and think of these as the (continuously) differentiable functions on $X$.
Inductively define $\dom(\Delta^{k})$ for each $k$ and then $\dom(\Delta^{\infty})=\cap_{k}
\dom(\Delta^{k})$. We say $f$ is smooth if $f\in\dom(\Delta^{\infty})$.  Harmonic functions have
zero Laplacian.

By introducing a normal derivative $\partial_{n}$ at boundary points the defining equation for the
Laplacian can be extended to functions that do not vanish on $V_{0}$. As a result we have the
Gauss-Green formula $\DF(u,v)= -\int (\Delta u)v\,d\mu + \sum_{x\in V_{0}}v(x)\partial_{n}u(x)$
when $v\in\dom(\DF)$, as in Theorem~3.7.8 of~\cite{Kigamibook}.  This formula may be localized to a
cell $F_{w}(X)$, in which case $\partial_{n}^{w}u(q)=\lim_{m}\DF(u,v_{m})$ at the boundary point
$q=\bigcap_{m} F_{wr_{j}^{m}}(X)$, where $v_{m}$ is the harmonic function on $F_{wr_{j}^{m}}(X)$
with all boundary values equal to $0$ other than $v_{m}(q)=1$.  The superscript $w$ in
$\partial_{n}^{w}u(q)$ indicates which cell the normal derivative is taken with respect to, as
there is one for each cell that intersects at $q$.  In general the normal derivatives exist once
$\Delta u$ exists as a measure. If $u\in\dom(\Delta)$ then the normal derivatives at a point sum to
zero. Conversely, if $u$ is defined piecewise by giving functions $u_{j}\in\dom(\Delta)$ each
supported on one of the cells that share the boundary point $x$, then $u\in\dom(\Delta)$ if and
only if all $u_{j}(x)$ are equal, all $\Delta u_{j}(x)$ are equal, and the normal derivatives of
the $u_{j}$ at $x$ sum to zero. We call these constraints the {\em matching conditions} for the
Laplacian.

\subsection*{Resistance Metric}

In addition to the Laplacian and other derviatives, the Dirichlet form also provides us with a
metric intrinsic to the fractal. We define the resistance metric $R$ by
\begin{equation*}
    R(x,y) = \min  \{E(u)^{-1}: u\in\dom(E),\ u(x)=0,\ u(y)=1\}.
    \end{equation*}
In Sections 2.3 and 3.3 of \cite{Kigamibook}  it is proven that under our assumptions this minimum
exists and defines a metric, and that the $R$-topology coincides with the topology induced from the
embedding of $X$ into $\mathbb{R}^{d}$.  Of particular importance for us is the fact that
continuity may be treated using the resistance metric, for which purpose the following
H\"{o}lder-$\frac{1}{2}$ estimate is very useful:
\begin{equation}\label{setting_resistmetricestimate}
    |u(x)-u(y)|^{2}
    \leq E(u)R(x,y)
    \leq \|u\|_{2}\|\Delta u\|_{2} R(x,y).
    \end{equation}
If $u\in\dom(\Delta)$ vanishes on $V_{0}$ then we obtain the first inequality trivially from the
definition and the second by applying the Cauchy-Schwartz inequality to $E(u)=-\int u\Delta u\,
d\mu$.  For general $u\in\dom{\Delta}$ we can simply subtract the harmonic function with the same
boundary values and apply the same estimate.  In particular, this shows that the $L^{2}$ domain of
the Laplacian embeds in the continuous functions.

\subsection*{Fractafolds}
Since the results in this paper are primarily local in nature, we will be able to work on a
connected fractafold based on $X$ with a restricted cellular construction, which we denote by
$\domain$. Some results on fractafolds and their spectra may be found in \cite{MR1990573}. As with
a manifold based on Euclidean space, a fractafold based on $X$ is just a connected Hausdorff space
in which each point has a neighborhood homeomorphic to a neighborhood of a point in $X$. One way to
construct a fractafold is by suitably gluing together copies of $X$, for example by identifying
appropriate boundary points. This leads us to the idea of a cellular construction, which is the
analogue of a triangulation of a manifold. A restricted cellular construction consists of a finite
or countably infinite collection of copies $X_{j}$ of $X$, together with an admissible
identification of their boundary points. Admissibility expresses the requirement that the result of
the gluing be a fractafold; more precisely, it means that if $\{x_{1},\dotsc,x_{J}\}$ are
identified then there is a junction point $x\in X$ and a neighborhood $U$ of $x$ such that each of
the components $U_{1},\dotsc,U_{J}$ of $U\setminus \{x\}$ is homeomorphic to a neighborhood of the
corresponding point $x_{j}$ in $X_{j}$. We call any such point $x$ a {\em gluing point}, and make
the obvious definition that a neighborhood of $x$ is a union of neighborhoods of $x$ in each of the
cells $X_{j}$ that meet at $x$ in the manner previously described.

It should be noted that the above is not the most general kind of cellular construction (hence the
term {\em restricted} in the definition), because some fractals have non-boundary points (called
terminal points) at which cells may be glued (see \cite{MR1990573}, Section 2). Dealing with such
points introduces certain technicalities that, while not insurmountable, cause complications in
defining the Green's operator (see below) that we will need for proving Theorem
\ref{structure_structuretheorem}.  It is worth noting that if $X$ has some topological rigidity
then all fractafolds have restricted cellular structure.  This is true, for example, for
fractafolds based on the Sierpinski Gasket (\cite{MR1990573} Theorem 2.1).

Thus far our fractafold has only topological structure; however if $\domain$ has a restricted
cellular construction then a smooth structure may be introduced in the same manner as it was on $X$
itself, specifically by defining a Dirichlet energy and a measure and thus a weak Laplacian. We can
take the energy on $\domain$ to be the sum of the energies on the cells of the cellular
construction, and the measure (which is not necessarily finite, but is finite on compacta) to be
the sum of the measures on the cells:
\begin{align*}
    E(u,v) &= \sum_{j} E_{X_{j}}(u\evald{X_{j}},v\evald{X_{j}})
    = \sum_{j} a_{j} E_{X}(u\evald{X_{j}}\circ\iota_{j},v\evald{X_{j}}\circ\iota_{j})\\
    \mu(A) &= \sum_{j} \mu_{X_{j}}(A\cap X_{j})=\sum_{j} b_{j} \mu_{X}(\iota_{j}^{-1}(A\cap X_{j}))
    \end{align*}
where $\iota_{j}:X\to X_{j}$ is the map from the cellular construction. In the same way that the
angle sum at a vertex of a triangulation of a manifold determines the curvature at the vertex, the
choice of the weights $a_{j}$ and $b_{j}$ amount to a choice of metric on $\domain$, with
$a_{j}=b_{j}=1$ for all $j$ being the flat case (see \cite{MR1990573}, Section 6).  As all of the
computations made later in the paper may be made on one cell at a time, we will henceforth suppress
the weights $a_{j}$ and $b_{j}$.

Well-known results about p.c.f. fractals imply the existence of a Green's function (for which there
is an explicit formula) on finite unions of cells in a fractafold with cellular construction.
\begin{lemma}\label{setting_greensfnexsts}
Let $K=\cup_{1}^{J} X_{j}$ be a connected finite union of cells in $\domain$ and such that $K\neq
\domain$. Then there is a Green's operator $G_{K}$ with the property that if $\nu$ is a Radon
measure on $K$ (i.e. a Borel measure that is finite on compacta, outer regular on Borel sets and
inner regular on open sets), then $G_{K}\nu$ is continuous, $-\Delta G_{K}\nu = \nu$ on the
interior of $K$, and $G_{K}\nu\evald{\partial K}=0$.  The same conclusion holds in the case
$K=\domain$ under the additional assumption $\int d\nu=0$.
\end{lemma}
\begin{remark}
It is clear that $\partial K$ is a subset of the boundary points of the cells $X_{j}$, specifically
consisting of those gluing points at which not all glued cells are included in $K$.
\end{remark}

\begin{proof}
We recall from Sections 3.6 and 3.7 of \cite{Kigamibook} that our assumptions on $X$ imply there is
a Green's operator $G$ on $X$ with continuous kernel $g(x,y)$, such that $-\Delta G\nu = \nu$ and
$G\nu\evald{\partial X}=0$ for all Radon measures $\nu$.  There is an explicit formula giving
$g(x,y)$ as a series.

If $G_{j}$ is the Green's operator for the cell $X_{j}$ it is easy to verify that $-\Delta\sum
G_{j}\nu = \nu$, except at the gluing points where the Laplacian can differ from $\nu$ by Dirac
masses, the size of which may be computed explicitly by summing the normal derivatives of the
$G_{j}\nu$ at the points that are glued.  However it is also apparent that by assigning values at
each of the gluing points and extending harmonically on the cells we obtain a continuous and
piecewise harmonic function, the Laplacian of which is a sum of Dirac masses at the gluing points.

Provided the boundary $\partial K$ is non-empty (which is obvious if $K\neq\domain$), a linear
algebra argument (Lemma 3.5.1 in \cite{Kigamibook}) shows that for any prescribed set of weights
for Dirac masses of the Laplacian at interior gluing points, there is a unique piecewise harmonic
function that vanishes on $\partial K$ and has this Laplacian. Subtracting this piecewise harmonic
function from $\sum G_{j}\nu$ gives the required $G_{K}\nu$.

On the other hand, if $K=\domain$ then $\domain$ is compact and the kernel of $\Delta$ is precisely
the constant functions.  We can therefore invert $-\Delta$ on the measures that annihilate
constants, that is, those for which $\int d\nu=0$.  This can be done explicitly in the same manner
as in the previous case, except that the linear algebra step now shows the Laplacians of piecewise
harmonics span the space of mean-zero linear combinations of Dirac masses at the gluing points.  In
this case the choice of piecewise harmonic function is unique up to adding a constant; our
convention is to choose this constant so $\int G_{K}\nu (x) d\mu(x)=0$.
\end{proof}

Throughout this paper we will assume that $\domain$ has no boundary.  In some examples it is
possible to deal with boundary points by passing to an appropriate cover, but relatively little is
known in terms of covering theory for general fractafolds. Elementary examples to keep in mind
include non-compact cases like open subsets of $X\setminus V_{0}$ or the infinite Sierpinski Gasket
treated in \cite{MR1658094}, and compact fractafolds like the double cover of the Sierpinski Gasket
$SG$, which consists of two copies of $SG$ with each boundary point from one copy identified with
exactly one boundary point of the other (see \cite{MR1990573} for more details).

\subsection*{Smooth Cutoffs and Partitioning}
As mentioned earlier, the structure theorems we shall prove for distributions rest heavily on
results from \cite{RST}. In what follows we assume that $\domain$ is a fractafold with restricted
cellular structure and is based on a fractal $X$ with regular harmonic structure.

Recall that $x\in X$ is a junction point if and only if there is a neighborhood $U\ni X$ such that
$U\setminus\{x\}$ is disconnected into a finite number of components $U_{j}$.  For a smooth
function $u$ the quantities $\Delta^{k}u(q)$ and $\partial_{n}^{j}\Delta^{k}u(q)$ exist for all
$k\in\mathbb{N}$; the superscript $j$ on $\partial_{n}^{j}$ indicates the normal derivative with
respect to the cell $U_{j}$.  For a fixed $j$, the two sequences $\Delta^{k}u(q)$ and
$\partial_{n}^{j}\Delta^{k}u(q)$ make up the {\em jet} of $u$ at $q$ in $U_{j}$.  The first result
we need from \cite{RST} is a Borel-type theorem on the existence of smooth functions with
prescribed jets.

\begin{theorem}[\protect{\cite{RST}}, Theorem 4.3 and Equation 4.8]\label{setting_Borelthm}
Given values $\rho_{0},\rho_{1},\dotsc$ and $\sigma_{0},\sigma_{1},\dotsc$ there is a smooth
function $f$ on $U_{j}$ that vanishes in a neighborhood of all boundary points except $q$, where
the jet is given by $\Delta^{k}f(q)=\rho_{k}$ and $\partial_{n}^{j}\Delta^{k}f(q)=\sigma_{k}$ for
all $k$. If we write $U_{j}$ as $U_{j}=F_{w}(X)$ for a word $w$, and fix a number $L$ of jet terms,
then for any $\epsilon>0$ we may construct $f$ so that for $0\leq k\leq L$, we have the estimate
\begin{equation}\label{setting_jetestimate}
    \|\Delta^{k}f\|_{\infty}
    \leq C(k) (r_{w}\mu_{w})^{-k} \biggl( \sum_{l=0}^{L} r_{w}^{l}\mu_{w}^{l} |\rho_{l}| + \sum_{l=0}^{L-1}
    r_{w}^{l+1}\mu_{w}^{l} |\sigma_{l}| \biggr) +\epsilon
\end{equation}
where $C(k)$ depends only on $k$ and the harmonic structure on $X$.
\end{theorem}

\begin{remark}
Of course, it follows immediately that we can construct a smooth function with prescribed jets at
each of the boundary points of a cell $K$ and an estimate like \eqref{setting_jetestimate}, just by
applying the theorem separately to each of the boundary points and summing the result.
\end{remark}

\begin{corollary}\label{setting_bumpfunctionwithestimates}
If $K$ is a cell in $\domain$ and $U$ is an open neighborhood of $K$, then there is a smooth
function $f$ such that $f=1$ on $K$, $f=0$ outside $U$, and $\|f\|_{\infty}\leq C$, where $C$ is a
constant that does not depend on $K$ or $U$.
\end{corollary}

\begin{proof}
Let $\{q_{j}\}$ be the boundary points of $K$ and at each $q_{j}$ take cells $V_{j,k}\subset U$
such that $\bigcup_{k} V_{j,k} \cup K$ contains a neighborhood of $q_{j}$. By making all of these
cells sufficiently small and removing any inside $K$ we may further assume that the $V_{j,k}$
intersect $K$ only at $q_{j}$, intersect each other only at $q_{j}$ and do not intersect
$V_{j',k'}$ for any $j'\neq j$.

On each $V_{j,k}$ construct the smooth function $f_{j,k}$ guaranteed by Theorem
\ref{setting_Borelthm} with $f_{j,k}=1$ at $q_{j}$ and all other jet terms at $q$ equal to zero,
and taking $\epsilon=1$. Then the piecewise function
\begin{equation*}
    f(x) = \begin{cases}
        1 &\quad\text{for $x\in K$}\\
        f_{j,k} &\quad\text{for $x\in V_{j,k}$}\\
        0 &\quad\text{otherwise}
        \end{cases}
    \end{equation*}
is equal $1$ on $K$ and $0$ off $U$ by construction.  It is also smooth, simply because the pieces
are smooth and the matching conditions for $\Delta^{l}$ apply at each of the boundary points of the
$V_{j,k}$ for all $l$.  The bound $\|f\|_{\infty}\leq C$ independent of $K$ and $U$ now follows
from \eqref{setting_jetestimate} because the scale-dependent terms are all raised to the power
zero, so are constant.
\end{proof}

A more difficult task than that in Corollary~\ref{setting_bumpfunctionwithestimates} is to
construct a {\em positive} bump function that is equal to $1$ on $K$ and to zero outside the
neighborhood $U$ of $K$.  A result of this type was proven in \cite{RST} under certain assumptions
on the diffusion $Y_{t}$ corresponding to the Laplacian.  A sufficient assumption is that the heat
kernel $p_{t}(x,y)$  (i.e. the transition density of of $Y_{t}$) satisfies an estimate of the form
\begin{equation}\label{setting-subGaussianbound}
    p_{t}(x,y)
    \leq \frac{\gamma_{1}}{t^{\alpha/\beta}} \exp\, \biggl( -\gamma_{2} \Bigl( \frac{R(x,y)^{\beta}}{t} \Bigr)^{1/(\beta-1)}
    \biggr)
    \end{equation}
where $\alpha$, $\beta$, $\gamma_{1}$ and $\gamma_{2}$ are constants. The
estimate~\ref{setting-subGaussianbound} is known to be valid in great generality on p.c.f. fractals
(Corollary~1.2 of \cite{MR1665249}).
\begin{theorem}[\protect{\cite{RST} Corollary~2.9}]\label{setting_positivebumpfunction}
Under the assumption~\eqref{setting-subGaussianbound}, for a cell $K$ and an open neighborhood
$U\supset K$, there is a smooth function $f$ such that $f=1$ on $K$, $f=0$ outside $U$, and
$f(x)\geq 0$ for all $x$.
\end{theorem}

The final theorem from \cite{RST} that we will use extensively is concerned with the smooth
partitioning of a smooth function.

\begin{theorem}[\protect{\cite{RST}}, Theorem 5.1]\label{setting_partitionthm}
Let $K\subset X$ be compact and fix $\bigcup U_{\alpha}\supset K$ an open cover.   If $f\in
\dom(\Delta^{\infty})$ then there is a decomposition $f=\sum_{1}^{J}f_j$ in which each $f_j$ is in
$\dom(\Delta^{\infty})$ and has support in some $U_{\alpha_{j}}$.
\end{theorem}
\begin{remark}\label{setting_partitionremarkone}
Compactness of $K$ is used only to obtain finiteness of the decomposition, and may be omitted for
finite covers.  An analogous countable (and locally finite) decomposition is then valid in the
$\sigma$-compact case; in particular it is valid on $\domain$, because of the existence of a
cellular structure.
\end{remark}

\begin{remark}\label{setting_partitionremarktwo}
The proof uses a result on the existence of smooth functions with prescribed jet at a point
(Theorem 4.3 of \cite{RST}) to smoothly join cutoffs to a piece of the original function as in the
proof of Theorem \ref{setting_bumpfunctionwithestimates}. This is very different from the Euclidean
case where one simply multiplies the smooth function by a smooth bump. In particular, the
construction of the cutoff depends explicitly on the growth rate of the jet of $f$ at the boundary
points under consideration, so for a collection of sets indexed by $j$, the mapping $f\mapsto
\{f_{j}\}$ to a sequence of smooth functions supported on these sets is nonlinear.

Although the non-linearity will make some later proofs more complicated, this method does provide
good estimates. From~\eqref{setting_jetestimate} and standard arguments for controlling the normal
derivative $\partial_n\Delta^{k}f$ at a point by the norms $\|\Delta^{j}\|_{\infty}$,
$j=0\dotsc,k+1$, over a neighborhood of the point (like those in Section
\ref{distribsatpoint_section} below) we find that $f\mapsto f_{j}$ can be arranged to satisfy
\begin{equation}\label{setting_extensionestimate}
    \|\Delta^{k} f_{j}\|_{\infty}\leq C \sum_{l=0}^{k} \|\Delta^{l}f\|_{\infty}
    \end{equation}
where $C$ is a constant depending only on $k$ and $K$.
\end{remark}

\section{Test Functions} \label{testfunction_section}
We define test functions on $\domain$ in the usual way, and provide notation for the space of
smooth functions on $\domain$ topologized by uniform convergence on compacta.
\begin{definition} \label{testfunction_testfndefn}%
The space of {\em test functions} $\testfns$ consists of all $\testfn\in\dom(\Delta^{\infty})$ such
that $\sppt(\testfn)$ is compact.  We endow it with the topology in which $\testfn_i \to \testfn$
iff there is a compact set $K\subset\domain$ containing the supports of all the $\testfn_i$, and
$\Delta^k \testfn_i \to \Delta^{k}\testfn$ uniformly on $K$ for each $k\in\mathbb{N}$. There is a
corresponding family of seminorms defined by
\begin{equation}
    \label{testfunction_seminormeqn}
    |\testfn|_m = \sup \{|\Delta^k \testfn(x)| : x \in \Omega, k \leq m\}
  \end{equation}
though it should be noted that the topology on $\testfns$ is not the usual metric topology produced
by this family.  For a discussion of the topology on $\testfns$ and its relation to these
seminorms, see Chapter~6 of \cite{Rudin}.
\end{definition}
\begin{definition} \label{testfunction_smoothfndefn}%
$\smoothfns=\dom(\Delta^{\infty})$ with the topology $\testfn_i \to \testfn$ iff for every compact
$K\subset\Omega$ we have $\Delta^k \testfn_i \to \Delta^{k}\testfn$ uniformly on $K$ for each
$k\in\mathbb{N}$. There is a corresponding family of seminorms defined by
\begin{equation}
    \label{testfunction_seminormeqnforsmoothfns}
    |\testfn|_{m,K} = \sup \{|\Delta^k \testfn(x)| : x \in K, k \leq m\}.
  \end{equation}
\end{definition}
The following result is immediate from Theorem \ref{setting_partitionthm} and
\eqref{setting_extensionestimate}. It will be used frequently in the results proved below.
\begin{lemma}
  \label{testfunction_partition}
  If $\testfn \in \mathcal{D}(\domain_1 \cup \domain_2)$, then $\testfn = \testfn_1 + \testfn_2$
  for some $\testfn_j \in \mathcal{D}(\domain_j)$.  For each $m$ there is $C=C(M,\Omega_1,\Omega_2)$ so
  $|\testfn_{j}|_{m}\leq C|\testfn|_{m}$, $j=1,2$.
\end{lemma}
One consequence is that $\testfns$ is dense in $\smoothfns$, because we may fix an increasing
compact exhaustion $\cup_{j}K_{j}=\domain$ of our domain and for arbitrary $\testfn\in\smoothfns$
write $\testfn=\testfn_{j}+\tilde{\testfn}_{j}$, where $\testfn_{j}$ is supported in $K_{j+1}$ and
$\tilde{\testfn}_{j}$ is supported in $K_{j}^{c}$, so that
$\testfn\evald{K_{j}}=\testfn_{j}\evald{K_{j}}$. The functions $\testfn_{j}$ are in $\testfns$ and
it is clear that $\Delta^{k}\testfn_{j}\to \Delta^{k}\testfn$ uniformly on compacta, hence
$\testfn_{j}\to\testfn$ in $\smoothfns$.  Another density result that follows from
Lemma~\ref{testfunction_partition} is as follows.

\begin{theorem}\label{testfunction_testfnsdenseincts}
\testfns\ is dense in $C_{c}(\domain)$, the space of continuous functions with compact support,
with supremum norm.
\end{theorem}
\begin{proof}
The dual of $C_{c}(\domain)$ is the space of Radon measures, so by the Hahn-Banach Theorem, it
suffices to show that if such a measure $\nu$ satisfies
\begin{equation}
    \label{testfunction_int-test-fn-zero}
    \int \testfn \,d\nu = 0, \quad \text{for all $\testfn \in \mathcal{D}(\domain)$},
    \end{equation}
then $\nu \equiv 0$.

Let $K$ be a cell and $\{U_i\}$ a sequence of open sets containing $K$ so that $\nu(U_i \setminus
K) \to 0$. Using Corollary \ref{setting_bumpfunctionwithestimates} we see that for each $j$ we can
take $\testfn_{i}\in\testfns$ with $\testfn_{i}\equiv1$ on $K$, the bound
$\|\testfn_{i}\|_{\infty}\leq C$ for all $i$, and $\sppt(\testfn_{i})\subset(U_{i})$. Then for
$\nu$ satisfying \eqref{testfunction_int-test-fn-zero} we compute
\begin{equation*}
    \nu (K)
    = \left| \int \testfn_i \,d\nu - \mu(K) \right|
    \leq \|\testfn_j\|_\infty \nu(U_j \setminus K)
    \leq C \nu(U_j \setminus K) \to 0.
    \end{equation*}
As $\nu$ vanishes on all cells it is the zero measure, and the result follows.
\end{proof}

Since $\domain$ is locally compact and Hausdorff, it is a standard result that $C_{c}(\domain)$ is
supremum-norm dense in $C_{0}(\domain)$, where the latter consists of those continuous functions
$f$ for which the set $\{x:|f(x)|\geq \epsilon\}$ is compact for all $\epsilon>0$.  Hence
$\testfns$ is also dense in $C_{0}(\domain)$.

 In the special case where $\Omega$ is compact we may
also characterize $\testfns=\smoothfns$ by the decay of the Fourier coefficients obtained when
$\testfn$ is written with respect to a basis of Laplacian eigenfunctions.  This provides an
alternate proof of the density of $\testfns$ in $C_{c}(\domain)$, which of course coincides with
$C(\Omega)$ in this case.

\begin{theorem}\label{testfunctions_compactcase}
If $\Omega$ is compact then $\testfns=\smoothfns$ is the space of smooth functions with Fourier
coefficients that have faster than polynomial decay, and hence is dense in $C(\Omega)$.
\end{theorem}
\begin{proof}
Clearly $\testfn\in\testfns$ is in $L^2$, so can be written $\testfn=\sum_{i=0}^{\infty}
a_{i}\psi_{i}$, where $\psi_{i}$ is the Laplacian eigenfunction with eigenvalue $-\lambda_{i}$.  It
follows that $(-\Delta)^{k}\testfn=\sum_{i} a_{i}\lambda_{i}^{k}\psi_{i}$ with convergence in
$L^{2}$. Since $\Delta^{k}\testfn$ is in $C(\Omega)\subset L^{2}$ for all $k$ we see that the
sequence $a_{i}$ must decay faster than any polynomial in the $\lambda_{i}$.  Conversely any such
sequence converges to a function for which every power of the Laplacian is in $L^{2}$, whereupon
the function is smooth by iteration of~\eqref{setting_resistmetricestimate}. In addition, any $u\in
C(\domain)$ can be explicitly approximated by functions from $\testfns$ by taking successive
truncations of the Fourier series $u=\sum_{i=0}^{\infty}a_{i}\psi_{i}$.  To see this gives
convergence in $\testfns$ write $(-\Delta)^{k}\sum_{i=j}^{\infty}a_{i}\psi_{i} =
\sum_{i=j}^{\infty}a_{i}\lambda_{i}^{k}\psi_{i}$ and note this converges to zero in $L^{2}$ and
therefore almost everywhere.  Now from~\eqref{setting_resistmetricestimate}
\begin{equation*}
    \sup_{\Omega} \Bigl| \sum_{i=j}^{\infty}a_{i}\lambda_{i}^{k}\psi_{i} \Bigr|^{2}
    \leq C \Bigl\| \sum_{i=j}^{\infty}a_{i}\lambda_{i}^{k}\psi_{i} \Bigr\|_{2}
    \Bigl\| \sum_{i=j}^{\infty}a_{i}\lambda_{i}^{k+1}\psi_{i} \Bigr\|_{2}
    \end{equation*}
and both terms on the right converge to zero.
\end{proof}

In~\cite{MR1962353} there is a definition of Sobolev spaces on p.c.f. fractals of the type studied
here. These spaces may be defined by applying the Bessel potential $(I-\Delta)^{-s}$ (for the
Dirichlet or Neumann Laplacian) or the Riesz potential $(-\Delta)^{-s}$ (for the Dirichlet
Laplacian) to $L^{p}$ functions on the fractal, and adding on an appropriate space of harmonic
functions.  In particular, the space of $L^{2}$ functions with $\Delta^{k}u\in L^{2}$ for $0\leq
k\leq m$ may be identified with a particular $L^{2}$ Sobolev space (\cite{MR1962353} Theorem~3.7).
Writing $W^{s,2}$ for the $L^{2}$ Sobolev space arising from $(I-\Delta)^{-s}$, we have in
consequence of the preceding:
\begin{corollary}\label{IntersectofSobolevspaces}
If $\Omega$ is compact then $\testfns=\cap_{s>0}W^{s,2}$.
\end{corollary}

\section{Distributions}\label{distributions_section}
\begin{definition}\label{distributions_defnofdistrib}
The space of {\em distributions} on $\domain$ is the dual space $\distribs$ of $\testfns$ with the
weak-star topology, so $\distrib_{i}\to\distrib$ if and only if
$\distrib_{i}\testfn\to\distrib\testfn$ for all $\testfn\in\testfns$.
\end{definition}
As usual, the most familiar examples of distributions are the Radon measures. If $d\nu$ is such a
measure then we define $\distrib_{\nu}$ by $\distrib_{\nu}\testfn=\int \testfn d\nu$. Theorem
\ref{testfunction_testfnsdenseincts} ensures that the mapping $\nu\mapsto\distrib_{\nu}$ is
injective, so we may identify $\nu$ and $T_{\nu}$.  One way to obtain further examples is to take
the adjoint of the Laplacian on distributions, which clearly produces another distribution.

\begin{definition}
If $T\in\distribs$ we define $\Delta T\in\distribs$ by $(\Delta T)\testfn = T(\Delta\testfn)$ for
all $\testfn\in\testfns$.
\end{definition}

It is clear that powers of the Laplacian applied to the Radon measures provide a rich collection of
examples of distributions.  Later we prove that all distributions arise in essentially this way
(Theorem \ref{structure_structuretheorem}), but we first need to establish some more elementary
properties.

\begin{theorem}
A linear functional $\distrib$ on $\testfns$ is a distribution if and only if for each compact
$K\subset\domain$ there are $m$ and $M$ such that
\begin{equation}\label{distributions_conditiontobedist}
    | \distrib \testfn |
    \leq M |\testfn|_{m}
    \end{equation}
\end{theorem}
\begin{proof}
It is clear that the existence of such an estimate ensures continuity of $\distrib$.  To prove the
converse we assume no such estimate exists, so there is $K$ compact and a sequence $\testfn_{j}$
such that $| \distrib\testfn_{j}|\geq j|\testfn_{j}|_{j}$.  Then the support of
$\tilde{\testfn}_{j}=\testfn_{j}/\distrib \testfn_{j}$ is in $K$ for all $j$ and
\begin{equation*}
    \Bigl\| \Delta^{k} \tilde{\testfn}_{j} \Bigr\|_{\infty}
    \leq \frac{ |\testfn_{j} |_{k} } {|\distrib \testfn_{j}|}
    \leq \frac{1}{j} \quad \text{once $j\geq k$.}
    \end{equation*}
Therefore $\tilde{\testfn}_{j}\to 0$ in $\testfns$ but has $\distrib \tilde{\testfn}_{j}=1$ for all
$j$, contradicting the continuity of $\distrib$.
\end{proof}

In the special case that $\domain$ is compact we saw in Theorem \ref{testfunctions_compactcase}
that $\testfns$ consists of smooth functions having Fourier coefficients that decay faster than
polynomially. This allows us to identify its dual with coefficient sequences having at most
polynomial growth.
\begin{lemma}\label{polygrowthlemma}
If $\Omega$ is compact and $T\in\distribs$ then the sequence $T\psi_{i}$ has at most polynomial
growth. Conversely, any sequence $\{b_{i}\}$ of polynomial growth defines a distribution via
$\testfn=\sum_{i}a_{i}\psi_{i}\mapsto\sum_{i}a_{i}b_{i}$.
\end{lemma}
\begin{proof}
We saw in Theorem~\ref{testfunctions_compactcase} that $\sum_{i=0}^{j}a_{i}\psi_{i}$ converges to
$\testfn$ in $\testfns$ if and only if $\{a_{i}\}$ has faster than polynomial decay in
$\lambda_{i}$.  It follows that for any $T\in\distribs$,
$\sum_{i=0}^{j}a_{i}T\psi_{i}=T\sum_{i=0}^{j}a_{i}\psi_{i}\rightarrow T\testfn$, from which the
sequence $T\psi_{i}$ has at most polynomial growth.

Conversely suppose that $\{b_{i}\}$ has polynomial growth, $|b_{i}|\leq C\lambda_{i}^{m}$, and
consider the map $\{a_{i}\}\mapsto \sum_{i}a_{i}b_{i}$.  This is a well defined linear map $T$ on
sequences $\{a_{i}\}$ with faster than polynomial decay, hence on $\testfns$, with the estimate
\begin{equation*}
    |T\testfn|
    \leq \sum_{i} |a_{i} b_{i}|
    \leq C \sum_{i} |a_{i}| \lambda_{i}^{m}
    = C\| \Delta^{m}\testfn\|_{2}
    \leq C \| \Delta^{m}\testfn \|_{\text{sup}}
    \leq C |\testfn |_{m}
    \end{equation*}
which shows that $T$ is a distribution.
\end{proof}

In particular, if we identify $\testfns$ as a subset of $\distribs$ by letting
$\testfn'\in\testfns$ act on $\testfns$ via $\testfn\mapsto \langle \testfn,\testfn'\rangle$, where
$\langle,\rangle$ is the $L^{2}$ inner product, then this implies that the test functions are dense
in the distributions when $\Omega$ is compact.  To see this, define $T_{j}$ by
\begin{equation*}
    T_{j}\psi_{i}
    =\begin{cases}
        T\psi_{i} &\text{if $i\leq j$}\\
        0 &\text{if $i>j$}.
        \end{cases}
    \end{equation*}
We see that $(T-T_{j})\testfn = \sum_{i=j}^{\infty} a_{i}b_{i}\rightarrow0$ for any
$\testfn\in\testfns$, so $T_{j}\rightarrow T$.  Since $T_{j}$ is the inner product with the
function $\sum_{i=0}^{j}\bar{b_{i}}\psi_{i}$, it is in $\testfns$. This is true more generally.

\begin{theorem}\label{distribs_testfnsdenseindistribs}
$\testfns$ is dense in $\distribs$.
\end{theorem}
\begin{proof}
Let $T\in\distribs$.  Take an increasing exhaustion $\cup K_{j}$ of $\Omega$ by compact sets
$K_{j}$ with the property that $K_{j}$ is contained in the interior of $K_{j+1}$, and each $K_{j}$
is a finite union of cells. For each $j$, the action of $T$ on $\testfnson(K_{j})$ identifies it as
an element of $\distribson(K_{j})$ so by the preceding there is a sequence
$\{T_{j,k}\}_{k=0}^{\infty}\subset\testfnson(K_{j})$ for which $T_{j,k}\rightarrow T$ in
$\distribson(K_{j})$, and hence in $\distribson(K_{l})$ for all $l\leq j$.

Now consider the diagonal sequence $T_{j,j}$. For any test function $\testfn$ there is some $j_{0}$
such that $K_{j_{0}}$ contains the support of $\testfn$, and hence $T_{j,j}\testfn$ is defined for
$j\geq j_{0}$ and converges to $T\testfn$.  So $T_{j,j}\rightarrow T$ in $\distribs$. Of course,
$T_{j,j}$ only corresponds to a test function $\testfn_{j}$ on $K_{j}$, not to an element of
$\testfns$. To remedy this, note that for the test function $\testfn_{j}$ corresponding to
$T_{j,j}$ on $K_{j}$ we may apply Theorem~\ref{setting_Borelthm} to each of the (finite number of)
boundary points of $K_{j}$ and thereby continue $\testfn_{j}$ smoothly to a function $\testfn_{j}'$
on $\Omega$  with support in $K_{j+1}$.  Denote by $T_{j}'$ the distribution corresponding to this
new test function $\testfn_{j}'$.  Since $\testfn_{j}$ and $\testfn_{j}'$ coincide on $K_{j}$ we
see that $T_{j}\testfn=T_{j}'\testfn$ for all $\testfn$ having support in $K_{j}$.  It follows that
$T_{j}'$ converges to $T$ in $\distribs$, and since each $T_{j}'$ corresponds to a test function,
the proof is complete.
\end{proof}

\begin{definition}
If $\domain_{1}\subset\domain$ is open, we say the distribution $\distrib$ {\em vanishes on
$\domain_{1}$} if $\distrib\testfn=0$ for all $\testfn$ supported on $\domain_{1}$.  This is
written $\distrib\evald{\domain_{1}}=0$,
\end{definition}

To make a meaningful definition of the support of a distribution we again need the partitioning
property. By Lemma~\ref{testfunction_partition} we know that any
$\testfn\in\testfnson(\domain_{1}\cup\domain_{2})$ can be written as
$\testfn=\testfn_{1}+\testfn_{2}$ for $\testfn_{j}\in\testfnson(\domain_{j})$. If both
$\distrib\testfn_{1}=0$ and $\distrib\testfn_{2}=0$ it follows that $\distrib\testfn=0$.  We record
this as a lemma, and note that it establishes the existence of a maximal open set on which
$\distrib$ vanishes.

\begin{lemma}
If $\distrib\evald{\domain_{1}}=0$ and $\distrib\evald{\domain_{2}}=0$ then
$\distrib\evald{\domain_{1}\cup\domain_{2}}=0$.
\end{lemma}

\begin{definition}
The {\em support} of $\distrib$ is the complement of the maximal open set on which $\distrib$
vanishes, and is denoted $\sppt(\distrib)$. In the special case where $\sppt(\distrib)$ is compact
we call $\distrib$ a {\em compactly supported distribution}.
\end{definition}

\begin{theorem}
The space of compactly supported distributions is (naturally isomorphic to) the dual
$\compactdistribs$ of the smooth functions on $\domain$.
\end{theorem}
\begin{proof}
The inclusion $\testfns\subset\smoothfns$ defines a natural map from $\compactdistribs$ to
$\distribs$. We have seen (after Lemma~\ref{testfunction_partition} above) that $\testfns$ is dense
in $\smoothfns$, from which it follows that the kernel of this map is trivial.  For convenience we
identify $\compactdistribs$ with its isomorphic image in $\distribs$, so we need only verify it is
the space of distributions with compact support.

Fix an increasing sequence of compacta $K_{j}$ with $\domain=\cup_{j}K_{j}$.  If
$\distrib\in\compactdistribs$ fails to be compactly supported then for each $j$ there is
$\testfn_{j}$ supported in $\domain\setminus K_{j}$ such that $\distrib\testfn_{j}\neq0$, and by
renormalizing we may assume $\distrib\testfn_{j}=1$ for all $j$.  However for any compact $K$ there
is $j$ such that $K\subset K_{j}$ and thus $\testfn_{l}\equiv0$ on $K$ once $l\geq j$.  This
implies $\testfn_{j}\to0$ in $\smoothfns$ and $\distrib\testfn_{j}=1$ for all $j$, a contradiction.

Conversely, let $\distrib\in\distribs$ be supported on the compact $K\subset\domain$.  Fix a
strictly larger compact $K_{1}$ (so that $K$ is contained in the interior of $K_{1}$) and an open
neighborhood $\domain_{1}$ of $K_{1}$. By Remark \ref{setting_partitionremarkone} the conclusion of
Theorem \ref{setting_partitionthm} is valid for the cover by $\domain\setminus K_{1}$ and
$\Omega_{1}$, even though $\domain$ is noncompact.  In particular we can fix a decomposition
mapping in which $f\in\smoothfns$ is written as $f=f_{1}+f_{2}$, with $f_{2}$ supported on
$\domain\setminus K_{1}$ and therefore $f_{1}\evald{K_{1}}\equiv f\evald{K_{1}}$.  Now let
$\distrib_{1}$ on $\smoothfns$ be given by $\distrib_{1}f=\distrib f_{1}$.  This is well defined,
because if $f,g\in\smoothfns$ and $f_{1}=g_{1}$, then $f-g\equiv0$ in a neighborhood of $K$ and the
support condition ensures $\distrib f=\distrib g$.  It is also linear, even though the mapping
$f\mapsto f_{1}$ is nonlinear (see Remark \ref{setting_partitionremarktwo}), because
$(f+g)_{1}=f_{1}+g_{1}$ on $K_{1}$, which contains a neighborhood of $K$.  Lastly, $\distrib_{1}$
is continuous, as may be seen from the fact that a sequence $\{\testfn_{j}\}\subset\smoothfns$ such
that $\Delta^{k}\testfn_{j}\to0$ uniformly on compacta will have $\Delta^{k}(\testfn_{j})_{1}\to0$
on $K_{1}\supset K$, or from \eqref{setting_extensionestimate}.  We conclude that every compactly
supported distribution is in $\compactdistribs$.
\end{proof}


\section{Structure Theory}\label{structuresection}

\begin{definition}
A distribution $T$ has {\em finite order} $m$ if for each compact $K$ there is $M=M(K)$ such that
$|T\phi|\leq M|\phi|_{m}$ for all $\phi\in\testfnson(K)$.
    \end{definition}

The following theorem indicates the importance of the finite order distributions.

\begin{theorem}\label{structure_compactimpliesfiniteorder}
Compactly supported distributions have finite order.
\end{theorem}

\begin{proof}
Let $T$ be a distribution with compact support $K$ and let $K_1$ be a compact set such that
$K\subset\interior(K_{1})$. By Lemma~\ref{testfunction_partition} we may decompose any
$\testfn\in\testfns$ as $\testfn=\testfn_{1}+\testfn_{2}$ where $\testfn_{1}$ is supported on
$K_1$, $\testfn_2$ is supported in $K^{c}$, and $|\testfn_{j}|_{m}\leq C_{m} |\testfn|_{m}$ for
$j=1,2$. Clearly $T\testfn=T\testfn_{1}$, but there are $m$ and $M$ so that
\eqref{distributions_conditiontobedist} holds on $K_1$, from which we conclude that
\begin{equation*}
    |T\phi|=|T\phi_1|\leq M|\phi_1|_{m}\leq C_m M |\phi|_{m}.\qedhere
    \end{equation*}
\end{proof}

It is easy to see that the Radon measures on $\domain$ are examples of distributions of finite
order.  In fact they have order zero, because the action of $\nu$ on $\testfns$ via
$\nu\testfn=\int \testfn\nu$ trivially satisfies the bound
$|\nu\testfn|\leq\|\testfn\|_{\infty}=|\testfn|_{0}$. The converse is also true.

\begin{theorem}\label{structure-zeroorderimpiesmeasure}
If $T$ is a distribution of order zero then there is a Radon measure $\nu$ such that $T\testfn =
\int \testfn d\nu$ for all $\testfn\in\testfns$.
\end{theorem}

\begin{proof}
Let $K$ be compact.  Since $T$ has order zero there is $M=M(K)$ so that $|T\testfn|\leq
M\|\testfn\|_{\infty}$ whenever $\testfn\in\testfns$ has support in $K$.  Theorem
\ref{testfunction_testfnsdenseincts} shows that these functions are dense in $C(K)$, so we may
extend $T$ to a bounded linear operator on $C(K)$.  Such operators are represented by Radon
measures, so there is $\nu_{K}$ with $T\testfn=\int \testfn d\nu_{K}$ for all test functions
$\testfn$ with support in $K$.  Now let $\bigcup K_{j}$ be a compact exhaustion of $\domain$ and
consider the measures $\nu_{K_{j}}$.  These converge weak-star as elements of the dual of
$C_{c}(\domain)$ to a Radon measure $\nu$, but by construction $\int \testfn
d\nu_{K_{j}}\rightarrow T\testfn$ on $\testfns$ and the result follows.
\end{proof}

\begin{remark} As written, the preceding proof relies on Theorem \ref{testfunction_testfnsdenseincts} and hence
on the Hahn-Banch theorem.  Since each of the $K_{j}$ is compact, a constructive proof can be
obtained by instead using Theorem \ref{testfunctions_compactcase}.
\end{remark}

A well known application of the preceding is obtaining a characterization of the distributions that
have positive values on positive test functions.  To prove this we need Corollary
\ref{setting_positivebumpfunction}, and therefore must make the corresponding assumption
\eqref{setting-subGaussianbound} on the behavior of the heat kernel corresponding to the Laplacian.

\begin{definition}
$T$ is a positive distribution if $T\testfn\geq0$ whenever $\testfn\geq0$ is a positive test
function.
\end{definition}

\begin{theorem}\label{structure_positivedistribsarepositivemeasures}
Positive distributions have order zero.  If the Laplacian on $X$ is such that
\eqref{setting-subGaussianbound} holds and if $T$ is a positive distribution, then there is a
positive measure $\nu$ such that $Tf=\int fd\nu$.
\end{theorem}

\begin{proof}
Let $K$ be compact.  Using Theorem \ref{setting_bumpfunctionwithestimates} there is
$\psi_{K}\in\testfns$ such that $\psi\equiv1$ on $K$.  If $\testfn\in C^{\infty}$ with support in
$K$ then the functions $\|\testfn\|_{\infty}\psi_{K}\pm\testfn$ are both positive, whence
\begin{equation*}
    -\|\testfn\|_{\infty} T \psi_{K}
    \leq T\testfn
    \leq \|\testfn\|_{\infty}\psi_{K}.
    \end{equation*}
We conclude that $T$ has order zero, so by Corollary~\ref{structure-zeroorderimpiesmeasure} it is
represented by integration against a measure $\nu$.  If there is a cell $K$ for which $\nu(K)<0$
then we can take $U_{j}$ to be a neighborhood of $K$ for which $\nu(U_{j}\setminus K)<1/j$ and let
$f_{j}$ be as in Theorem~\ref{setting_positivebumpfunction}.  It follows that
\begin{equation*}
    Tf_{j} =\int f\, d\nu \leq \nu (K) + \frac{1}{j}\|f\|_{\infty}
    \end{equation*}
and for a sufficiently large $j$ this is negative, in contradiction to the positivity of $T$.  We
conclude that $\nu(K)\geq 0$ for all cells $K$, and therefore that $\nu$ is a positive measure.
\end{proof}

We noted at the beginning of Section \ref{distributions_section} that the adjoint of the Laplacian
maps $\distribs$ to itself.  In particular, if $\nu$ is a Radon measure, hence a distribution of
order zero, then for each compact $K$ there is $M(K)$ such that
\begin{equation*}
    |(\Delta^{m}\nu) \testfn|
    =|\nu (\Delta^{m}\testfn)|
    \leq M(K) |\Delta^{m}\testfn|_{0}
    \leq M(K) |\testfn|_{m}
    \end{equation*}
so $\Delta^{m}\nu$ is a distribution of order $m$.  This result has a converse, which we prove
using a modification of the Green's function introduced in Lemma \ref{setting_greensfnexsts}.  The
basic idea is to produce a Green's operator that inverts the Laplacian on test functions, so that
the adjoint of this operator lowers the order of a finite-order distribution. Iterating to produce
a distribution of zero order then produces a measure by
Theorem~\ref{structure-zeroorderimpiesmeasure}.

\begin{lemma}\label{structure_greensfnforKexists}
Let $K$ be a connected finite union of cells in $\domain$.  Then
$\Delta:\testfnson(K)\rightarrow\testfnson(K)$ and its image consists of all test functions that
are orthogonal to the harmonic functions on $K$.  Moreover there is a linear operator
$\tilde{G}_{K}:\testfnson(K)\rightarrow\testfnson(K)$ such that
$-\Delta\tilde{G}_{K}(\Delta\testfn)=\Delta\testfn$ for all $\testfn\in\testfnson(K)$.
\end{lemma}

\begin{proof}
For $\testfn\in\testfnson(K)$, the matching conditions for the Laplacian ensure that both $\testfn$
and $\partial_{n}\testfn$ vanish on $\partial K$ because $\testfn$ is identically zero outside $K$.
If $f\in\dom(\Delta)$, then the Gauss-Green formula reduces to
\begin{equation*}
    \int_{K} (\Delta \testfn) f \, d\mu  = \int_{K} \testfn (\Delta f) \, d\mu
    \end{equation*}
because there are no non-zero boundary terms.  We conclude that $f$ is orthogonal to the image
$\Delta\bigl(\testfnson(K)\bigr)$ if and only if $\Delta f$ is orthogonal to $\testfnson(K)$.  As
the latter is dense in $C(K)$ (Lemma \ref{testfunction_testfnsdenseincts}) the first result is
proven.

Let $h_{1}\dotsc,h_{i(K)}$ be an $L^{2}$--orthonormal basis for the finite dimensional space of
harmonic functions on $K$.  As $\testfnson(K)$ is dense in $C(K)$ in supremum norm and $K$ has
finite measure, $\testfnson(K)$ is also dense in both $C(K)$ and $L^{2}(K)$ in $L^{2}$ norm.  It
follows that there are $\testfn_{1},\dotsc,\testfn_{i(K)}$ in $\testfnson(K)$ such that
$\langle\testfn_{i},h_{j}\rangle=\delta_{ij}$, where $\langle,\rangle$ is the $L^{2}$ inner product
and $\delta_{ij}$ is Kronecker's delta.  Given $\testfn\in\testfnson(K)$ we let
\begin{equation}\label{structure_defnoftestfnK}
    \tilde{\testfn} = \sum_{i=1}^{i(K)} \langle\testfn,h_{i}\rangle \testfn_{i}
    \end{equation}
and define
\begin{equation*}
    \tilde{G}_{K} \testfn (x)=
    \begin{cases}
    G_{K} (\testfn-\tilde{\testfn})(x) &\quad\text{if $x\in K$}\\
    0 &\quad\text{if $x\not\in K$}
    \end{cases}
    \end{equation*}
where $G_{K}$ is the Green's operator defined in Lemma \ref{setting_greensfnexsts}.  It is then
clear that for $\psi\in\testfnson(K)$,
\begin{equation}\label{structure_laplacianofGtilde}
    -\Delta\tilde{G}_{K} \psi
    = -\Delta G_{K} (\psi -\tilde{\psi})
    =\psi -\tilde{\psi}
    \end{equation}
except perhaps at points of $\partial{K}$, where we must first verify that the matching conditions
for the Laplacian hold.  Since $\tilde{G}_{K}\psi$ vanishes outside $K$, the matching conditions
require that $\partial_{n}\tilde{G}_{K}\psi (q)=0$ whenever $q\in\partial K$.  One way to verify
this is from the Gauss-Green formula for a harmonic function $h$, which yields
\begin{equation*}
    0
    = \langle \psi-\tilde{\psi},h \rangle
    =\int_{K} (-\Delta\tilde{G}_{K}\psi) h
    = -\sum_{q\in\partial K} \bigl(-\partial_{n}\tilde{G}_{K}\psi(q)\bigr) h(q)
    \end{equation*}
from which we see that it suffices to know the solvability of the Dirichlet problem on $K$, that
is, for every assignment of boundary values on $\partial K$ there is a harmonic function $h$ with
those boundary values. This latter is true because of Lemma \ref{setting_greensfnexsts}; for
example it may be proven by taking a function that is piecewise harmonic on cells and has the
desired boundary data and subtracting the result of applying $G_{K}$ to its Laplacian (which is
simply a sum of Dirac masses at the interior gluing points).  We conclude that $\tilde{G}_{K}
\psi\in\testfnson(K)$ and that \eqref{structure_laplacianofGtilde} holds everywhere.

Finally, if $\psi=\Delta\testfn$ for some $\testfn\in\testfnson(K)$, then $\psi$ is orthogonal to
the harmonics, so $\tilde{\psi}=0$ and $-\Delta\tilde{G}_{K}\psi=\psi$ as desired.
\end{proof}

The adjoint of $\tilde{G}_{K}$ is defined on distributions by
$(\tilde{G}_{K}T)\testfn=T(G_{K}\testfn)$. This operator is really defined on the dual of
$\testfnson(K)$, which is a larger space, but we will not make use of this fact.

\begin{theorem}\label{structure_Gondistribus}
If $T$ is a distribution of order $m\geq 1$ then $\tilde{G}_{K}T$ is a distribution of order $m-1$,
and if $T$ is a distribution of order zero then $\tilde{G}_{K}T$ is integration with respect to a
continuous function on $K$.
\end{theorem}

\begin{proof}
Let $T$ be a distribution of order $m\geq 1$, so that $\tilde{G}_{K}T\testfn$ is bounded by
\begin{equation*}
    |\tilde{G}_{K}\testfn|_{m}
    = \sup \{ \|\Delta^{k} \tilde{G}_{K}(\testfn) \|_{\infty} : k\leq m \}.
    \end{equation*}
When $k\geq 1$ we have $\Delta^{k} \tilde{G}_{K}(\testfn) =
-\Delta^{k-1}(\testfn-\tilde{\testfn})$, and when $k=0$ we see that $\|
\tilde{G}_{K}(\testfn)\|_{\infty}\leq C \|\testfn - \tilde{\testfn}\|_{\infty}$ because the
operator $G_{K}$ in Lemma \ref{setting_greensfnexsts} is clearly bounded on $L^{\infty}$.  Hence
$|\tilde{G}_{K}\testfn|_{m}\leq C|\testfn-\tilde{\testfn}|_{m-1}\leq C(m-1,K)|\testfn|_{m-1}$,
where the latter inequality is from \eqref{structure_defnoftestfnK} with a constant $C(m-1,K)$ that
may depend on the set of functions $\testfn_{j}$. Thus $\tilde{G}_{K}T$ has order $m-1$.

If $T$ has order zero then by Theorem \ref{structure-zeroorderimpiesmeasure} it is represented by
integration against a Radon measure $\nu$.  Provided $K\neq \domain$ we can apply Lemma
\ref{setting_greensfnexsts} directly to see $\nu=\Delta f$ for some $f$ that is continuous on $K$
and vanishes on $\partial K$, so can be extended continuously to be zero outside $K$.  This ensures
there are no boundary terms when we compute with the Gauss-Green formula:
\begin{align*}
    \tilde{G}_{K}T\testfn
    &=T\tilde{G}_{K}\testfn
    =\int_{K} \tilde{G}_{K} \testfn \, d\nu
    =\int_{K} (\tilde{G}_{K}\testfn)( \Delta f) \, d\mu\\
    &=\int_{K} (-\Delta\tilde{G}_{K}\testfn) f\, d\mu
    =\int_{K} (\testfn - \tilde{\testfn} ) f\, d\mu\\
    &=\int_{K} \Bigl( \testfn - \sum_{i=1}^{i(K)} \langle\testfn,h_{i}\rangle \testfn_{i} \Bigr)
    f\, d\mu\\
    &=\int_{K} \testfn \Bigl(f - \sum_{i=1}^{i(K)} \langle \testfn_{i},\bar{f} \rangle \bar{h}_{i}
    \Bigr) \, d\mu
    \end{align*}
and the bracketed term in the last line is continuous because $f$ is continuous and all of the
$h_{i}$ are harmonic.

The argument is slightly different if $\domain=K$.  We instead set $t=\int d\nu/(\int d\mu)$ so
$\int d(\nu -t\mu)=0$, at which point Lemma \ref{setting_greensfnexsts} applies to show $\nu-t\mu =
\Delta f$ for a continuous $f$, and we can compute as before
\begin{align*}
    \tilde{G}_{K}T\testfn
    &= \int_{K} \tilde{G}_{K}\testfn \, d\nu\\
    &= t\int_{K} \tilde{G}_{K}\testfn  \, d\mu + \int_{K} \tilde{G}_{K}\testfn ( \Delta f) \, d\mu\\
    &= \int_{K} (\testfn - \tilde{\testfn} ) f\, d\mu\\
    &= \int_{K} \testfn \Bigl( f-\int_{K} f\, d\mu \Bigr) \, d\mu
    \end{align*}
where we used that $\int_{K} \tilde{G}_{K}\testfn\, d\mu = 0$ (from the proof of Lemma
\ref{setting_greensfnexsts}) and that the harmonic functions are constants in this case.
\end{proof}

We now have all the necessary tools to prove the main structure theorem for distributions.

\begin{theorem}\label{structure_structuretheorem}
Any distribution $T$ may be written as a locally finite sum of the form $T=\sum
\Delta^{m_{j}}\nu_{j}$ or $T=\sum \Delta^{m_{j}+1} f_{j}$, where the $\nu_{j}$ are Radon measures
and the $f_{j}$ are continuous functions with compact support.
\end{theorem}

\begin{proof}
Suppose first that $\domain$ is non-compact and take $K_{1}, K_{2}, \dotsc$ a sequence of subsets
such that each $K_{j}$ is a connected finite union of cells, $K_{j}$ is contained in the interior
of $K_{j+1}$, and $\cup_{j} K_{j} = \domain$. Such a sequence exists because $\domain$ has a
restricted cellular construction. It will be convenient to also set $K_{0}=\emptyset$. For each $j$
let $\tilde{G}_{j}=\tilde{G}_{K_{j}}$ be the operator from
Lemma~\ref{structure_greensfnforKexists}. The key point of the proof is that for any distribution
$S$, we have $(-\Delta)^{m}\tilde{G}_{j}^{m}S=S$ as elements of $\distribson(K_{j})$ (though not as
elements of $\distribs$). This may be verified by direct computation.  For all
$\testfn\in\testfnson(K_{j})$,
\begin{equation}\label{structure_laplacianinvertsGondistribs}
    -\Delta \tilde{G}_{j} S \testfn
    = (-\tilde{G}_{j}S) (\Delta \testfn)
    = S \bigl(-\tilde{G}_{j} (\Delta \testfn) \bigr)
    = -S \testfn
    \end{equation}
where the final step uses that $-\Delta \tilde{G}_{j} (\Delta \testfn)=\Delta\testfn$ from Lemma
\ref{structure_greensfnforKexists}, so $\tilde{G}_{j} (\Delta \testfn)+\testfn$ is harmonic on $K$
and vanishes on $\partial K$, hence is identically zero.

Fix $T\in\distribs$.  Inductively suppose that for $i=0,\dotsc,j-1$ we have found $m_{i}$ and a
measure $\nu_{i}$ supported on $K_{i}$ such that $T-\sum_{0}^{j-1} \Delta^{m_{i}}\nu_{i}$ vanishes
on $\testfnson(K_{j-1})$.  The base case $j=0$ is trivial because $K_{0}=\emptyset$.  Now
$T-\sum_{0}^{j-1} \Delta^{m_{i}}\nu_{i}$ is in $\distribs$, hence its restriction to
$\testfnson(K_{j})$ is in $\distribson(K_{j})$.  We call this restriction $T_{j}$.  As $K_{j}$ is
compact, $T_{j}$ has finite order $m_{j}$.  It satisfies
$T_{j}=(-\Delta)^{m_{j}}\tilde{G}_{j}^{m_{j}}T_{j}$ by the argument already given, and by
Theorem~\ref{structure_Gondistribus} there is a measure $\nu_{j}$ supported on $K_{j}$ such that
$\nu_{j}=(-1)^{m_{j}}\tilde{G}_{j}^{m_{j}}T_{j}$.  Therefore $T_{j}=\Delta^{m_{j}}\nu_{j}$ in
$\distribson(K_{j})$, which is equivalent to saying that $T-\sum_{0}^{j} \Delta^{m_{i}}\nu_{i}$
vanishes on $\testfnson(K_{j})$.

It is immediate from the definition that $\sum_{j} \Delta^{m_{j}}\nu_{j}$ is a locally finite sum.
If we fix $\testfn\in\testfns$ then there is a $j$ such that $\testfn\in\testfnson(K_{j})$,
whereupon $\bigl(T-\sum_{i=1}^{l} \Delta^{m_{i}}\nu_{i}\bigr)\testfn=0$ for all $l\geq j$.  This
proves that $T=\sum_{j} \Delta^{m_{j}}\nu_{j}$.

The proof that $T=\sum_{j} \Delta^{m_{j}+1}f_{j}$ is similar.  Obviously we wish to use the latter
part of Theorem~\ref{structure_Gondistribus} to go from the measure $\nu_{j}$ to a continuous
function.  The only technicality is that the resulting $f_{j}$ is continuous on $K_{j}$ rather than
on all of $\domain$. We fix this at each step of the induction as follows.  Suppose we have
determined $T_{j}$ as the restriction of $T-\sum_{i=0}^{j-1}\Delta^{m_{i}+1}f_{i}$ and from
Theorem~\ref{structure_Gondistribus} a function $g_{j}$ continuous on $K_{j}$ such that
$\Delta^{m_{j}+1}g_{j}=T_{j}$ in $\distribson(K_{j})$. Let $f_{j}$ be a continuous extension of
$g_{j}$ to $\domain$ obtained by requiring $f_{j}=0$ on $\partial K_{j+1}$ and outside $K_{j+1}$,
and letting it be piecewise harmonic on the cells of the cellular structure on $K_{j+1}\setminus
K_{j}$ (here we use that $K_{j}$ is in the interior of $K_{j+1}$). Clearly
$\Delta^{m_{j}+1}f_{j}=T_{j}$ in $\distribson(K_{j})$, because we $f_{j}=g_{j}$ on $K_{j}$, so
$T-\sum_{i=0}^{j}\Delta^{m_{i}+1}f_{i}$ vanishes on $\testfnson(K_{j})$ and we may complete the
proof as before.

In the case when $\domain$ is compact the proof is somewhat more elementary because we need only a
single set $K=\domain$, but there is a small technical difference due to the fact that the final
equality of \eqref{structure_laplacianinvertsGondistribs} is no longer true.  Indeed,
$\tilde{G}_{j} (\Delta \testfn)+\testfn$ is harmonic by the same reasoning as for the non-compact
case, but now it is the possibly non-zero constant $\tilde{\testfn}=(\mu(K))^{-1}\int \testfn\,
d\mu$. The analogue of \eqref{structure_laplacianinvertsGondistribs} is therefore
\begin{equation*}
    \Delta \tilde{G}_{K} S \testfn
    = S \testfn - \tilde{\testfn} S 1
    \end{equation*}
where $1$ is the constant function.

The distribution $T$ has finite order $m$, and $T\testfn= T1\tilde{\testfn} +  \Delta\tilde{G}_{K}
T \testfn$.  Iterating, we have $\tilde{G}_{K}T\testfn = \tilde{G}_{K}T1 \tilde{\testfn} + \Delta
\tilde{G}_{K}^{2} T \testfn$, and then
\begin{equation*}
    T\testfn = \Delta^{L}\tilde{G}_{K}^{L}T\testfn + \biggl(\sum_{l=1}^{L} a_{l}\biggr) \tilde{\testfn}
    \end{equation*}
where $a_{l}=\tilde{G}_{K}^{l}T 1$.  If $L=m$ then Theorem \ref{structure_Gondistribus} implies the
first term is $\Delta^{m}\nu$, where $\nu$ is a measure, and if $L=m+1$ this term is
$\Delta^{m+1}f$, where $f$ is a continuous function.  In either case the second term is a constant
multiple of the measure $\mu$, or equivalently the constant (hence continuous) function $1$, so the
proof is complete.
\end{proof}

\section{Distributions supported at a point}\label{distribsatpoint_section}

A distribution with support a point $q$ is of finite order by Theorem
\ref{structure_compactimpliesfiniteorder}, and simple modifications of the arguments in Theorem
\ref{structure_structuretheorem} show that it is a power of the Laplacian applied to a measure with
support in a neighborhood of $q$.  The purpose of this section is to identify it more precisely as
a finite sum of certain derivatives of the Dirac mass at $q$; in general these derivatives are not
just powers of the Laplacian, but instead reflect the local structure of harmonic functions at $q$.

Identification of a distribution $T$ of order $m$ supported at $q$ is achieved by describing a
finite number of distributions $T_{j}$, $j=1,\dotsc,J$ with the following property: if
$\testfn\in\testfns$ has $T_{j}\testfn=0$ for all $j$ then for any $\epsilon>0$ there is a
neighborhood $U_{\epsilon}$ of $q$ and a decomposition $\testfn=\testfn_{q}+ (\testfn-\testfn_{q})$
into test functions such that $|\testfn_{q}|_{m}<\epsilon$ and $\testfn-\testfn_{q}$ vanishes on
$U_{\epsilon}$. The reason is that then $T\testfn=T\testfn_{q}$ because of the support condition,
and $|T\testfn_{q}|\leq C\epsilon$, from which we conclude that $T$ vanishes whenever all $T_{j}$
vanish.  It follows from a standard argument (for example, Lemma~3.9 of \cite{Rudin}) that $T$ is a
linear combination of the $T_{j}$.

The argument described in the previous paragraph motivates us to find conditions on a test function
$\testfn$  that ensure we can cut if off outside a small neighborhood of a point $q$ while keeping
the norm $|\testfn_{q}|_{m}$ of the cutoff small.  In order to proceed we will need some notation
for a neighborhood base of $q$.  If $q$ is a non-junction point then it lies in a single copy of
$X$ in the cellular structure, and within this copy there is a unique word $w$ such that
$F_{w}(X)=q$. The cells containing $q$ are then of the form $U_{i}=F_{[w]_{i}}(X)$.  For junction
points the situation is different, as $q$ can be the intersection point of several copies of $X$,
or can be in a single copy but be given by $F_{w_{j}}(X)=q$ for a finite number of words
$w_{1},\dotsc ,w_{J}$. We will not distinguish between these possibilities but will instead make
the convention that the distinct words determining $q$ may be used to distinguish copies of $X$ if
necessary.  With this assumed, let $U_{i,j}=F_{[w_{j}]_{i}}(X)$, and $U_{i}=\cup_{j}U_{i,j}$.

Fix $q$ and let $G_{i,j}$ denote the Dirichlet Green's operator on the cell $U_{i,j}$, omitting the
$j$ index if $q$ is a non-junction point.  If $\testfn\in\testfns$ we can then decompose $\testfn$
on $U_{i,j}$ into
\begin{equation*}
    \testfn\evald{U_{i,j}} = H_{i,j}\testfn + G_{i,j} \Delta\testfn
    \end{equation*}
where $H_{i,j}\testfn$ is the (unique) harmonic function on $U_{i,j}$ whose values on $\partial
U_{i,j}$ coincide with those of $\testfn\evald{U_{i,j}}$. By induction we obtain
\begin{equation}\label{distribsatapoint_localtaylor}
    \testfn\evald{U_{i,j}} =  \sum_{l=0}^{m-1} G_{i,j}^{l} H_{i,j} \Delta^{l} \testfn \ +
    G_{i,j}^{m}\Delta^{m}\testfn\evald{U_{i,j}}
    \end{equation}
and write $h_{i,j}^{l}= G_{i,j}^{l} H_{i,j} \Delta^{l}\testfn$.

\begin{lemma}\label{distribsatapoint_estsforhijl}
In the decomposition \eqref{distribsatapoint_localtaylor} we have at each $x\in U_{i,j}$ and
$p\in\partial U_{i,j}$ that $\Delta^{k}h_{i,j}^{l}(x)=\partial_{n}\Delta^{k}h_{i,j}^{l}(p)=0$ if
$k>l$, while for $k\leq l$,
\begin{align*}
    \bigl|\Delta^{k}h_{i,j}^{l}\bigr| &\leq c(k,l) r_{[w_{j}]_{i}}^{l-k} \mu_{[w_{j}]_{i}}^{l-k}
        \bigl\|H_{i,j}\Delta^{l}\testfn\bigr\|_{L^{\infty}(U_{i,j})} \\
    \bigl|\partial_{n}\Delta^{l}h_{i,j}^{l}(p) \bigr| &\leq c(k,l)  r_{[w_{j}]_{i}}^{l-k-1}
    \mu_{[w_{j}]_{i}}^{l-k} \bigl\|H_{i,j}\Delta^{l}\testfn \bigr\|_{L^{\infty}(U_{i,j})}
    \end{align*}
\end{lemma}

\begin{proof}
The conclusions for the cases $k>l$ are immediate from the fact that $\Delta^{l}h_{i,j}^{l}$ is
harmonic, while the remaining estimates are derived from the fact that the Laplacian scales by
$r_{w}\mu_{w}$ on a cell $F_{w}(X)$ while the normal derivative scales by $r_{w}$.
\end{proof}

Our purpose in making the above definitions is that estimates on the functions $h_{i,j}^{l}$ are
precisely what is needed to ensure we can cut off a smooth function in the manner previously
described.

\begin{theorem}\label{distribsatapoint_decayofHijtestfnimpliesgoodextension}
If $\testfn$ is such that $\Delta^{m}\testfn(q)=0$ and
$\|H_{i,j}\Delta^{l}\testfn\|_{L^{\infty}(U_{i,j})}=o\Bigl(r_{[w_{j}]_{i}}\mu_{[w_{j}]_{i}}\Bigr)^{m-l}$
for $0\leq l\leq m-1$ as $i\rightarrow\infty$, then for all $\epsilon>0$ there is $\testfn_{q}$
such that $|\testfn_{q}|_{m}\leq \epsilon$ and $\testfn-\testfn_{q}$ is supported away from $q$.
\end{theorem}

\begin{proof}
We begin by constructing a neighborhood of $U_{i,j}$ by adjoining cells at each of the points
$p\in\partial U_{i,j}$.  At each $p$ we require finitely many such cells, and we choose them so as
to intersect $U_{i,j}$ only at $p$.  It will also be convenient to assume that these cells have
comparable scale to the $U_{i,j}$, in the sense that they have the form $F_{\tilde{w}}(X)$ for some
word with length $|\tilde{W}|\leq i+i_{0}$ for some constant $i_{0}$.  Let $K$ be one of the cells
adjoined at a point $p$, and let $n_{p}$ be the number of cells adjoined at $p$.  Using Theorem
\ref{setting_Borelthm} we define a smooth function $f_{K}$ on $K$ with jet
$\Delta^{k}f_{K}(p)=\Delta^{k}\testfn(p)$ and $\partial_{n}^{K}\Delta^{k}f_{K}(p) = -(1/n_{p})
\partial_{n}^{U_{i,j}}\Delta^{k}\testfn(p)$, and with vanishing jets at the other boundary points
of $K$.  Having done this for the set $\mathcal{K}$ of adjoined cells we see from the matching
conditions for the Laplacian that
\begin{equation*}
    \testfn_{q}(x) = \begin{cases}
        \testfn(x) &\quad\text{for $x\in U_{i}$}\\
        \sum_{K\in\mathcal{K}} f_{K} &\quad\text{for $x\in\bigcup_{K\in\mathcal{K}} K$}\\
        0 &\quad\text{otherwise}
        \end{cases}
    \end{equation*}
defines a test function with the property that $\testfn-\testfn_{q}=0$ on $U_{i}$.

We must estimate $|\testfn_{q}|_{m}$.   There is an easy estimate for $\Delta^{k}\testfn$ for
$k\leq m$ from Lemma \ref{distribsatapoint_estsforhijl}:
\begin{align}
    |\Delta^{k}\testfn|
    &\leq \sum_{l=0}^{m-1} \bigl| \Delta^{k} h_{i,j}^{l} \bigr|  + \bigl|G_{i,j}^{m-k}\Delta^{m}\testfn \bigr| \notag\\
    &\leq \sum_{l=k}^{m-1} c(k,l)  r_{[w_{j}]_{i}}^{l-k} \mu_{[w_{j}]_{i}}^{l-k} \bigl\| H_{i,j}\Delta^{l}\testfn
    \bigr\|_{L^{\infty}(U_{i,j})}
        + c(k,m)r_{[w_{j}]_{i}}^{m-k} \mu_{[w_{j}]_{i}}^{m-k} \bigl\|\Delta^{m}\testfn
    \bigr\|_{L^{\infty}(U_{i,j})}\notag\\
    &\leq \sum_{l=k}^{m} c(k,l) o\bigl( r_{[w_{j}]_{i}}^{m-k} \mu_{[w_{j}]_{i}}^{m-k} \bigr) \notag\\
    &= o\bigl( r_{[w_{j}]_{i}}^{m-k} \mu_{[w_{j}]_{i}}^{m-k} \bigr)
    \label{distribsatapoint_testfnbyhijl}
    \end{align}
where we used $\Delta^{m}\testfn(q)=0$ to obtain that $\Delta^{m}\testfn(q)=o(1)$ when
$i\rightarrow\infty$. As a result we have good control of $|\testfn_{q}|_{m}$ on $U_{i,j}$.

A similar calculation allows us to estimate the size of the normal derivative
$|\partial_{n}\Delta^{k}\testfn(p)|$ at any of the points $p$ where pieces $f_{K}$ are attached. We
compute
\begin{align}
    |\partial_{n}\Delta^{k}\testfn(p)|
    &\leq  \sum_{l=0}^{m-1} \bigl| \partial_{n}\Delta^{k} h_{i,j}^{l} \bigr|
        + \bigl|\partial_{n}G_{i,j}^{m-k}\Delta^{m}\testfn \bigr| \notag\\
    &\leq \sum_{l=k}^{m-1}  c(k,l)  r_{[w_{j}]_{i}}^{l-k-1} \mu_{[w_{j}]_{i}}^{l-k} \bigl\| H_{i,j}\Delta^{l}\testfn
    \bigr\|_{L^{\infty}(U_{i,j})} + c(k,m)r_{[w_{j}]_{i}}^{m-k-1} \mu_{[w_{j}]_{i}}^{m-k}
    \bigl\|\Delta^{m}\testfn \bigr\|_{L^{\infty}(U_{i,j})}\notag\\
    &\leq \sum_{l=k}^{m} c(k,l) o\bigl( r_{[w_{j}]_{i}}^{m-k-1} \mu_{[w_{j}]_{i}}^{m-k} \bigr) \notag\\
    &= o\bigl( r_{[w_{j}]_{i}}^{m-k-1} \mu_{[w_{j}]_{i}}^{m-k} \bigr).
    \label{distribsatapoint_normalderivoftestfnbyhijl}
    \end{align}

Fix $K\in\mathcal{K}$ and examine $f_{K}$.  By assumption $K=F_{\tilde{w}}(X)$, so by
\eqref{setting_jetestimate} with the fixed number of jet terms $m$ we know
\begin{equation}\label{distribsatapoint_estimateofjetoffK}
    \|\Delta^{k}f_{K}\|_{\infty}
    \leq C(k)
    \biggl(\sum_{k'=0}^{m} r_{\tilde{w}}^{k'-k}\mu_{\tilde{w}}^{k'-k} \bigl|\Delta^{k'}\testfn(p)\bigr|
    + \sum_{k'=0}^{m-1}
        r_{\tilde{w}}^{k'+1-k}\mu_{\tilde{w}}^{k'-k} \bigl|\partial_{n}^{U_{i,j}}\Delta^{k'}\testfn(p) \bigr|\biggr)
    +\epsilon
    \end{equation}
provided $0\leq k\leq m$.  The terms involving $\bigl|\Delta^{k'}\testfn(p)\bigr|$ may be replaced
by the estimate \eqref{distribsatapoint_testfnbyhijl}.  For the terms involving normal derivatives
we use that $\partial_{n}\Delta^{k'}h_{i,j}^{k'}(p)=(1/n_{p})\partial_{n}\Delta^{k'}\testfn$ and
\eqref{distribsatapoint_normalderivoftestfnbyhijl}. The result is
\begin{align}
    \|\Delta^{k}f_{K}\|_{\infty}
    &\leq C(k)
    \biggl( \sum_{k'=0}^{m} o\bigl( r_{\tilde{w}}^{k'-k}\mu_{\tilde{w}}^{k'-k}  r_{[w_{j}]_{i}}^{m-k'}
         \mu_{[w_{j}]_{i}}^{m-k'} \bigr)
    + \sum_{k'=0}^{m-1} o\bigl( r_{\tilde{w}}^{k'+1-k}\mu_{\tilde{w}}^{k'-k}
     r_{[w_{j}]_{i}}^{m-k'-1} \mu_{[w_{j}]_{i}}^{m-k'} \bigr)
      \biggr) +\epsilon\notag\\
    &\leq o\bigl( r_{[w_{j}]_{i}}^{m} r_{\tilde{w}}^{-k} \mu_{[w_{j}]_{i}}^{m} \mu_{\tilde{w}}^{-k}
    \bigr) \sum_{k'=0}^{m} \biggl( \frac{r_{\tilde{w}}\mu_{\tilde{w}}}{r_{[w_{j}]_{i}}\mu_{[w_{j}]_{i}}}\biggr)^{k'}
        \biggl( 1 + \frac{r_{\tilde{w}}}{r_{[w_{j}]_{i}}} \biggr) +\epsilon.
        \label{distribsatapoint_comparablesizedcellstep}
    \end{align}
However, $\tilde{w}$ and $[w_{j}]_{i}$ have comparable length and are adjacent, so they differ only
in the final $i_{0}$ letters and therefore the ratios $r_{\tilde{w}}r_{[w_{j}]_{i}}^{-1}$ and
$\mu_{\tilde{w}}\mu_{[w_{j}]_{i}}^{-1}$ are bounded by constants depending only on $i_{0}$ and the
harmonic structure and measure.  It follows that
\begin{equation*}
    \|\Delta^{k}f_{K}\|_{\infty}
    \leq C(m,k,r,\mu) o\bigl(  r_{[w_{j}]_{i-1}} \mu_{[w_{j}]_{i-1}} \bigr)^{m-k}
    \end{equation*}
and combining this estimate for each $K\in\mathcal{K}$ with \eqref{distribsatapoint_testfnbyhijl}
proves that
\begin{equation*}
    \| \Delta^{k} \testfn_{q} \|_{L^{\infty}} = o\bigl(  r_{[w_{j}]_{i-1}} \mu_{[w_{j}]_{i-1}} \bigr)^{m-k}
    \quad \text{as $i\rightarrow\infty$}
    \end{equation*}
for $0\leq k\leq m$. In particular we can make $|\testfn_{q}|_{m}<\epsilon$ by making $i$
sufficiently large.
\end{proof}

Theorem \ref{distribsatapoint_decayofHijtestfnimpliesgoodextension} suggests that the natural
candidates for the distributions supported at $q$ are appropriately scaled limits of the maps
$\testfn\mapsto H_{i,j}\Delta^{l}\testfn$ as $i\rightarrow\infty$.  The question of how to take
such limits has been considered by a number of authors
\cite{MR1025071,MR1761365,MR2333477,pelanderteplyaev07,MR2417415}, and is generally quite
complicated. At the heart of this complexity is the fact that the local behavior of smooth
functions in a neighborhood of a point $q$ depends strongly (in fact almost {\em entirely}) on the
point $q$ rather than the function itself.  This property -- often called ``geography is destiny''
-- contrasts sharply with the Euclidean situation where neighborhoods of points are analytically
indistinguishable.  Its immediate implication for the structure of distributions with point support
is that the nature of these distributions must depend on the point in question.  In order of
increasing complexity we consider three cases: junction points, periodic points and a class of
measure-theoretically generic points.

\subsection*{Junction Points}

As before, the junction point $q$ is $q=F_{w_j}(X)$ for words $w_1,\dotsc w_J$, each of which
terminates with an infinite repetition of a single letter. The distributions corresponding to
approaching $q$ through the sequence $[w_j]_{i}$ may be understood by examining the eigenstructure
of the harmonic extension matrices $A_{i_{j}}$, the definition of which appeared in the Harmonic
Functions part of Section~\ref{setting_section}.

For notational convenience we temporarily fix one contraction $F$, let $A$ be the corresponding
harmonic extension matrix, and suppose $q$ is $\cap F^{i}(X)$. Let $r$ and $\mu$ be the resistance
and measure scalings of $F$, and $\gamma_1,\dotsc,\gamma_{n}$ be the eigenvalues of $A$, ordered by
decreasing absolute value, with eigenspaces $E_{1},\dotsc,E_{n}$. Of course $\gamma_{1}=1$ and
$E_{1}$ is the constant functions. Let $H_{i}u$ be the harmonic function on $F^{i}(X)$ that equals
$u$ on $\partial F^{i}(X)$, and $P_{s}$ be the projection onto $E_{s}$.  In what follows, $G$ is
the Dirichlet Green's operator on $X$ and $G_{i}$ is the same on $F^{i}(X)$.

\begin{definition}\label{distribsatapoint_defnofdsatjunctionpt}
Inductively define derivatives $d_{s}$ and differentials $D^{k}$, $k\geq1$ at the point $q$ by
setting $D^{0}u=u(q)$, and for each $s$ such that $(r\mu)^{k}< \gamma_{s} \leq (r\mu)^{k-1}$
\begin{equation}
    d_{s}u = \lim_{i\rightarrow\infty} \gamma_{s}^{-i} P_{s} H_{i} \bigl( u - G
    D^{k-1} \Delta u  \bigr)
    \end{equation}
if these limits exist.  Note that $d_{s}$ always exists for harmonic functions as the sequence is
constant in this case. Provided the necessary $d_{s}u$ exist we then let
\begin{equation}\label{defnofdifferential}
    D^{k}u=h+ G D^{k-1}\Delta u
    \end{equation}
where $h$ is the unique harmonic function on $X$ with $d_{s}h=d_{s}u$ for those $s$ with
$\gamma_{s}>(r\mu)^{k}$ and $d_{s}h=0$ for all other $s$.  We will also make use of $\bar{D}^{k}u$,
where $\bar{D}^{0}=u(0)$ and
\begin{equation*}
    \bar{D}^{k}u=\bar{h}+ G \bar{D}^{k-1}\Delta u
    \end{equation*}
where $\bar{h}$ is harmonic on $X$ with $d_{s}\bar{h}=d_{s}u$ for those $s$ with
$\gamma_{s}\geq(r\mu)^{k}$ and $d_{s}\bar{h}=0$ for all other $s$.
\end{definition}

\begin{lemma}\label{distribsatapoint_existenceofdsatjunctionpoints}
For $u\in\dom(\Delta^{k})$ and each $s$ with $\gamma_{s}>(r\mu)^{k}$ the derivative $d_{s}u$
exists, and
\begin{equation}\label{esttoprovedsuisorderkdistrib}
    |d_{s}u |\leq C(k)\sum_{l=0}^{k} \|\Delta^{l} u\|_{\infty}.
    \end{equation}
The differential satisfies
\begin{equation}\label{differentialestimate}
    \bigl\| u- D^{k}u \bigr\|_{L^{\infty}(F^{i}(X))} \leq  C(k) i^{k}  (r\mu)^{ki}
    \|\Delta^{k}u\|_{\infty},
    \end{equation}
and if we further suppose that $\Delta^{k}u\in\dom(E)$ then
\begin{equation}\label{differentialestimatetwo}
    \bigl\| u- \bar{D}^{k}u \bigr\|_{L^{\infty}(F^{i}(X))}
    \leq C(k)  (r\mu)^{ki} r^{i/2} E^{1/2}(\Delta^{k}u).
    \end{equation}
\end{lemma}

\begin{proof}
The proof is inductive.  When $k=0$ there are no $s$ with $\gamma_{s}>1=(r\mu)^{0}$, so the first
statement is vacuous and \eqref{differentialestimate} is immediate.  Suppose both hold up to $k-1$.

Write $u-GD^{k-1}\Delta u$ as $H_{0}u+G(\Delta u - D^{k-1}\Delta u)$, from which
\begin{equation}\label{decompofdsuviadifferential}
    d_{s}u
    =d_{s}H_{0}u + \lim_{i} \gamma_{s}^{-i}P_{s}H_{i} \bigl( G\Delta u - G D^{k-1}\Delta u \bigr)
    \end{equation}
provided the latter limit exists.   On the cell $F^{i}(X)$,
\begin{equation*}
    G (\Delta u- D^{k-1}\Delta u )
    = H_{i} G (\Delta u- D^{k-1}\Delta u ) + G_{i} (\Delta u- D^{k-1}\Delta u )
    \end{equation*}
thus
\begin{equation*}
    H_{i+1} G (\Delta u- D^{k-1}\Delta u )
    = A H_{i} G (\Delta u- D^{k-1}\Delta u ) + H_{i+1} G_{i} (\Delta u- D^{k-1}\Delta u ).
    \end{equation*}
In particular, if we project onto the eigenspace $E_{s}$ then the action of $A$ is multiplication
by $\gamma_{s}$.  Scaling implies $G_{i} (\Delta u- D^{k-1}\Delta u )$ is bounded by
\begin{equation}\label{estimateforGiappliedtouminusdifferential}
    \bigl| G_{i} (\Delta u- D^{k-1}\Delta u )\bigl|
    \leq C (r\mu)^{i} \bigl\| \Delta u- D^{k-1}\Delta u \|_{L^{\infty}(F^{i}(X))}
    \leq CC(k-1) i^{k-1}  (r\mu)^{ik} \|\Delta^{k}u\|_{\infty}
    \end{equation}
and the action of $H_{i+1}$ and $P_{s}$ can only improve this estimate, so
\begin{align}
    \lefteqn{\gamma_{s}^{-(i+1)} \Bigl| P_{s} H_{i+1} G (\Delta u- D^{k-1}\Delta u )
    - \gamma_{s} P_{s} H_{i} G (\Delta u- D^{k-1}\Delta u ) \Bigr|}\qquad& \notag\\
    &\leq \gamma_{s}^{-(i+1)} \Bigl| G_{i} (\Delta u- D^{k-1}\Delta u ) \Bigr|\notag\\
    &\leq C C(k-1) i^{k-1} \Bigl( \frac{r^k\mu^{k}}{\gamma_{s}}\Bigr)^{i}
    \|\Delta^{k}u\|_{\infty}
    \label{provingthesequenceforderivsisCauchy}
    \end{align}
This shows $\{\gamma_{s}^{-i} P_{s} H_{i} G (\Delta u- D^{k-1}\Delta u )\}$ is Cauchy when
$\gamma_{s}>(r\mu)^{k}$, and that its limit is bounded by $C(k) \|\Delta^{k}u\|_{\infty}$.  It
follows from \eqref{decompofdsuviadifferential} that $d_{s}u$ exists for these values of $s$, and
since $|d_{s}H_{0}u|\leq \|u\|_{\infty}$ we also obtain \eqref{esttoprovedsuisorderkdistrib}.

Summing the tail of \eqref{provingthesequenceforderivsisCauchy} establishes that
\begin{equation*}
    \bigl| d_{s}u - \gamma_{s}^{-i} P_{s}H_{i} (u-GD^{k-1}\Delta u) \bigr|
    \leq C(k) \Bigl(\frac{r^{k}\mu^{k}}{\gamma_{s}} \Bigr)^{i} \|\Delta^{k}u\|_{\infty}.
    \end{equation*}
Now let $h$ be the unique harmonic function with $d_{s}h=d_{s}u$ for those $s$ with
$\gamma_{s}>(r\mu)^{k}$ and $d_{s}h=0$ otherwise. Since $\gamma_{s}^{-i}P_{s}H_{i}h=d_{s}h$ is a
constant sequence, we find
\begin{equation}\label{orderestforprojofharmonicpieceofdifferential}
    \bigl| P_{s} H_{i} (u- h - G D^{k-1}\Delta u) \bigr|
    \leq C(k) (r\mu)^{ik} \|\Delta^{k}u\|_{\infty}
    \end{equation}
for those $s$ with $\gamma_{s}>(r\mu)^{k}$.  Recalling $D^{k}u=h+GD^{k-1}\Delta u$ from
\eqref{defnofdifferential} write
\begin{align}
    (u-D^{k}u)\evald{F^{i}(X)}
    &= H_{i} (u-D^{k}u) + G_{i} \bigl( \Delta (u-D^{k}u) \bigr) \notag\\
    &=  H_{i} (u-h - GD^{k-1}\Delta u ) + G_{i} \bigl( \Delta u- D^{k-1}\Delta u \bigr).
    \label{expansionofuminusDkuonFiX}
    \end{align}
We have estimated $G_{i} \bigl( \Delta u- D^{k-1}\Delta u \bigr)$ in
\eqref{estimateforGiappliedtouminusdifferential} and the terms $P_{s} H_{i} (u- h - G D^{k-1}\Delta
u)$ for $\gamma_{s}>(r\mu)^{k}$ in \eqref{orderestforprojofharmonicpieceofdifferential}.  What
remains are the terms $P_{s}H_{i} (u-h - GD^{k-1}\Delta u)$ for $\gamma_{s}\leq (r\mu)^{k}$. Each
of these is obtained as a sum, with
\begin{align}\label{estforlowerordertermsinexpansionwithdifferentialremoved}
    \bigl| P_{s}H_{i} (u-h - GD^{k-1}\Delta u) \bigr|
    &=\biggl| \sum_{j=0}^{i-1} \gamma_{s}^{i-j} P_{s} H_{j} G_{j} (u-h - GD^{k-1}\Delta u )
    \biggr| \notag\\
    &\leq CC(k-1) \sum_{j=0}^{i-1} \gamma_{s}^{i-j}  j^{k-1}  (r\mu)^{jk}
    \|\Delta^{k}u\|_{\infty} \notag\\
    &\leq CC(k-1) (r\mu)^{ik}   \|\Delta^{k}u\|_{\infty}
    \sum_{j=0}^{i-1} j^{k-1} \Bigl( \frac{\gamma_{s}}{(r\mu)^{k}} \Bigr)^{(i-j)k} \notag\\
    &\leq C(k) i^{k} (r\mu)^{ik} \|\Delta^{k}u\|_{\infty}
    \end{align}
because $\gamma_{s}\leq (r\mu)^{k}$.  This proves \eqref{differentialestimate} for $k$ and
completes the induction.

The proof of \eqref{differentialestimatetwo} uses essentially the same inductive argument with
$\bar{D}$ replacing $D$ and the estimate from \eqref{differentialestimatetwo} replacing that from
\eqref{differentialestimate} throughout. Note that in \eqref{provingthesequenceforderivsisCauchy}
we can have $\gamma_{s}\geq(r\mu)^{k}$ because there is an additional factor of $r^{i/2}$ so the
series still converges geometrically.  Also, in
\eqref{estforlowerordertermsinexpansionwithdifferentialremoved} the working is simplified because
for $\bar{D}$ we have these $\gamma_{s}<(r\mu)^{k}$ and the $r^{i/2}$ term is bounded, so the
convergence is geometric here also.  This allows us to remove the polynomial term in $i$.  The base
case $k=0$ is true because of the H\"{o}lder estimate~\eqref{setting_resistmetricestimate}.
\end{proof}

The map $d_{s}$ takes a smooth function to the eigenspace $E_{s}$.  We now fix orthonormal bases
for each of the $E_{s}$, and refer to the co-ordinates of $d_{s}$ with respect to the basis for
$E_{s}$ as the {\em components} of $d_{s}$; these components have values in $\mathbb{C}$.

\begin{corollary}\label{distribsatapoint_dsisadistrib}
Each component $d_{s,v}$ of a $d_{s}$ for which $(r\mu)^{k}<\gamma_{s}\leq (r\mu)^{k-1}$ is a
distribution supported at $q$ and of order at most $k$.  If $\gamma_{s}<(r\mu)^{k-1}$  then its
order is equal to $k$, and it is otherwise of order either $k-1$ or $k$. If $d_{s,v}$ is a
component that is a distribution of order $k$, then $\Delta^{l}d_{s,v}$ defined by
$\Delta^{l}d_{s,v}\phi=d_{s,v}\Delta^{l}\phi$ is also supported at $q$ and has order $k+l$.
\end{corollary}

\begin{proof}
It is apparent from the definition that $d_{s}$ is linear on $\testfns$ and that $d_{s}\testfn=0$
if $\testfn\in\testfns$ is identically zero in a neighborhood of $q$, so it follows from
\eqref{esttoprovedsuisorderkdistrib} that the components of $d_{s}$ are distributions of order at
most $k$ and are supported at $q$.

Suppose $\gamma_{s}<(r\mu)^{k-1}$ and let $u_{s,v}$ denote the harmonic function determined by the
eigenvector corresponding to $d_{s,v}$.  Then the values of $u_{s,v}$ are
$O(\gamma_{s})=o(r\mu)^{k-1}$ and $\Delta^{l}u=0$ for $l\leq 1$, so
Theorem~\ref{distribsatapoint_decayofHijtestfnimpliesgoodextension} implies that for any
$\epsilon>0$ there is a function $\psi$ equal to $u_{s,v}$ in a neighborhood of $q$ but with
$|\psi|_{k-1}<\epsilon$.  Since $d_{s,v}u_{s,v}=1$ and $d_{s,v}u_{s,v}=d_{s,v}\psi$ because of the
support condition, it cannot be that $d_{s,v}$ is order $k-1$ or less, so it has order $k$.

In the case $\gamma_{s}=(r\mu)^{k-1}<(r\mu)^{k-2}$ the above argument says that $d_{s,v}$ has order
at least $k-1$.  Both of the values $k-1$ and $k$ occur in examples.  For instance, when $k=1$, the
derivative $d_{1}u=u(q)$ corresponding to the constant harmonic functions has order $0=k-1$. A case
where there is a $d_{s}$ of this type with order $k$ occurs on the Sierpinski Gasket, see
Example~\ref{distribs_commentaboutSGcase} below.  This shows that scaling alone cannot identify the
order of $d_{s}$ when $\gamma_{s}=(r\mu)^{k-1}$.

The statement regarding $\Delta^{l}d_{s}$ is immediate.
\end{proof}

We now return to using the index $j$ to distinguish the words $w_{j}$ for which $x=F_{w_{j}}(X)$,
and accordingly denote by $d^{j}_{s}$ the derivative $d_{s}$ corresponding to the approach through
cells $F_{[w_{j}]_{i}}$.

\begin{theorem}\label{distribsatapoint_identifyingdistribsatjunctionpoint}
Let $T$ be a distribution of order $k$ supported at the junction point $q$, where $q=F_{w_{j}}(X)$,
$j=1,\dotsc,n$. The word $w_{j}$ terminates with infinite repetition of a letter which, by a
suitable relabeling we assume is $j$. Then $T$ is a finite linear combination of the distributions
$\Delta^{l}d^{j}_{s,v}$, for which $\gamma_{s}\geq(r_{j}\mu_{j})^{k-l}$.  The linear combination
runs over all such $s$, all basis elements $v$ for $E_{s}$, and all cells $j=1,\dotsc,n$ that meet
at $q$.
\end{theorem}

\begin{proof}
Suppose that $\testfn\in\testfns$ has the property that $\Delta^{l}d^{j}_{s,v}\testfn=0$ for all
$(r_{j}\mu_{j})^{k-l}\leq\gamma_{s}$.  It follows from
Definition~\ref{distribsatapoint_defnofdsatjunctionpt} that $\bar{D}^{k}\testfn=0$ and more
generally that $\bar{D}^{k-l}\Delta^{l}\testfn=0$ for all $l\leq k$.

However, the harmonic part of
$H_{i,j}\Delta^{l}\testfn=H_{i,j}\Delta^{l}(\testfn-\bar{D}^{k}\testfn)$ on the cell $U_{i,j}$ of
scale $i$ corresponding to the word $w_{j}$ is bounded by the maximum over the boundary vertices of
this cell, so from~\eqref{differentialestimatetwo}:
\begin{equation*}
    \bigl\| H_{i,j}\Delta^{l}\testfn \bigr\|_{L^{\infty}(U_{i,j})}
    =o(r_{j}\mu_{j})^{(k-l)i}
    =o\bigl(r_{[w_{j}]_{i}}\mu_{[w_{j}]_{i}}\bigr)^{k-l}
    \end{equation*}
for $0\leq l\leq k-1$.  We also have that $\Delta^{k}\testfn(q)=0$ because $\Delta^{l}\testfn(q)=
\Delta^{l}d^{j}_{1,v}\testfn=0$, so
Theorem~\ref{distribsatapoint_decayofHijtestfnimpliesgoodextension} shows that for any $\epsilon>0$
there is $\psi\in\testfns$ that is equal to $\testfn-\bar{D}^{k}\testfn$ in a neighborhood of $q$
and with $|\psi|_{k}<\epsilon$.

Using the support condition and the fact that $T$ has order $k$ yields
\begin{equation*}
    T\testfn
    = T\psi
    \leq M|\psi|_{k}
    < M\epsilon
    \end{equation*}
for some fixed $M$ depending only on $T$, and all $\epsilon>0$.  Thus $T\testfn=0$, and we have
shown that the kernel of $T$ contains the intersection of the kernels of the distributions
described. By a standard result (e.g. Lemma~3.9 of~\cite{Rudin}), $T$ is a linear combination of
these distributions.
\end{proof}

\begin{remark}\label{distribs_commentaboutSGcase}
Since $d_{1}$ corresponds to the eigenspace of constants, the distributions $d^{j}_{1}\Delta^{l}$
are independent of $j$ and are simply powers of the Laplacian applied to the Dirac mass at $x$.  It
should also be noted that for each $j$ the distribution $d^{j}_{2}$ corresponds to the largest
eigenvalue less than $1$, so gives the normal derivative at $x$ when approaching through the cells
$F_{[w_{j}]i}$, $i\rightarrow\infty$.  As a result $\sum_{j}d^{j}_{2}u=0$, and not all of these
distributions need appear in $T'$.

It should also be noted that the linear combination in
Theorem~\ref{distribsatapoint_identifyingdistribsatjunctionpoint} may include distributions of the
form $\Delta^{l}d^{j}_{s,v}$ having $\gamma_{s}=(r_{j}\mu_{j})^{k-l}$, and that it is possible for
these to be of order $k+1$. If this were to occur then we would have a non-trivial linear
combination of these $(k+1)$-order distributions such that the linear combination is of order only
$k$.  We do not know of an example in which this occurs, but cannot eliminate it as a possibility
because our arguments rely on scaling information.
\end{remark}

\begin{example}
The canonical example of a p.c.f.~self-similar fractal of the type we are describing is the
Sierpinski Gasket $SG$ with its usual symmetric harmonic structure (see \cite{Strichartzbook} for
details of all results described below).  In this case $r=3/5$ and $\mu=1/3$, so the Laplacian
scales by $r\mu=1/5$. Each of the harmonic extension matrices $A_{i}$ has eigenvalues $1$, $3/5$
and $1/5$ with one-dimensional eigenspaces. The corresponding derivatives at $q$ are
$d^{j}_{1}u=u(q)$ which is point evaluation, $d^{j}_{2}u(q)=\partial^{j}_{N}u(q)$ which is the
normal derivative at $q$ from the cell corresponding to $j$, and
$d^{j}_{3}u(q)=\partial^{j}_{T}u(q)$ which is the tangential derivative of $u$ at $q$ from this
cell.

There are two cells meeting at the junction point $q$.  Without loss of generality we suppose they
are indexed by $j=0,1$. The two corresponding normal derivatives $\partial^{j}_{N}u(q)$, $j=0,1$
satisfy the single linear relation that they sum to zero, and by
Corollary~\ref{distribsatapoint_dsisadistrib} they are of order $1$. The two tangential derivatives
$\partial^{j}_{T}u(q)$, $j=0,1$ are independent, and it is known that they cannot be controlled by
$\|u\|_{\infty}+\|\Delta u\|_{\infty}$ (see \cite{Strichartzbook}, page 60).  They are therefore of
order $2$.   It is also possible to see in this example that any non-trivial linear combination of
the $\partial^{j}_{T}u(q)$ has order less than $2$.  Writing $\delta_{q}$ for the Dirac mass at
$q$, we conclude from Theorem~\ref{distribsatapoint_identifyingdistribsatjunctionpoint}  that any
distribution $T$ of order $k$ at a junction point of $SG$ can be written as a linear combination of
the form
\begin{equation}
    T = \sum_{l=0}^{k} a_{l}\Delta^{l}\delta_{q}
        + \sum_{l=0}^{k-1} b_{l} \Delta^{l}\partial^{0}_{N}\delta_{q}
        + \sum_{l=0}^{k-2} \sum_{j=0,1} c_{l,j} \Delta^{l}\partial^{j}_{T}\delta_{q}.
    \end{equation}

This example also illustrates the issue described in the proof of
Corollary~\ref{distribsatapoint_dsisadistrib}, namely that there can be a $d_{s}$ with
$\gamma_{s}=(r\mu)^{k-1}$ and yet $d_{s}$ is order $k$.  In this case we have $\partial_{T}=d_{3}$
with $\gamma_{3}=1/5=r\mu$, so $k=2$, and $d_{3}$ is of order $2$.
\end{example}

\subsection*{Periodic and Eventually Periodic Points}

Periodic points are those $x=F_{w}(X)$ for which $w$ is a periodic word, meaning that $w$ is
composed of an infinite repetition of a fixed finite word $v$.  Eventually periodic points are
those for which the word $w$ is periodic after some finite number of letters. For these points
there is a theory similar to that used for junction points; we do not have to consider derivatives
corresponding to multiple cells, but instead of looking at the eigenstructure of a matrix $A_{i}$
we must examine that of $A_{v}$, which is a finite composition of the $A_{i}$ matrices. If
$\gamma_{s}$ is an eigenvalue of $A_{v}$ with eigenspace $E_{s}$, then we can define the derivative
$d_{s}$ as we did for junction points.  It is easy to see that the analogues of
Lemma~\ref{distribsatapoint_existenceofdsatjunctionpoints},
Corollary~\ref{distribsatapoint_dsisadistrib} and
Theorem~\ref{distribsatapoint_identifyingdistribsatjunctionpoint} all hold, simply by changing the
notation to refer to the infinitely repeated matrix being $A_{v}$, the eigenvalues $\gamma_{s}$
being those of $A_{v}$, and the Laplacian scaling factor to be $r_{v}\mu_{v}$ instead of
$r_{j}\mu_{j}$.

\subsection*{Generic Points}

We now consider a non-junction point $x=F_{w}(X)$, where $w=w_{1}w_{2}\ldots$ is an infinite word.
The behavior of harmonic functions on the cell $F_{[w]_{n}}(X)$ can be understood by considering
the product $A_{[w]_{n}}=\prod_{j=1}^{n}A_{w_{j}}$.  We need to understand their scaling
properties, for which we use the following approach from \cite{MR1761365}. Define for each unit
vector $\alpha$ the corresponding Lyupunov exponent
\begin{equation}\label{distribsatapoint_defnofLyupunov}
    \log \gamma(\alpha) = \lim \frac{1}{i} \log \| A_{[w]_{i}}\alpha \|
    \end{equation}
if the limit exists.  In this definition we may take $\|\cdot\|$ to be any norm on the $\#
V_{0}$-dimensional space containing $\alpha$; all such norms are equivalent, so $\gamma$ is
unaffected by this choice.

Let us suppose that these limits exist at $x$.  It is readily seen that
$\gamma(\alpha)\neq\gamma(\alpha')$ implies $\alpha$ and $\alpha'$ are orthogonal, so there are at
most $\# V_{0}$ distinct values $\gamma_{1}>\gamma_{2}\dotsm$ that occur. Corresponding to these is
a direct sum decomposition $E_{1}\oplus E_{2}\oplus\dotsm$ with the property that writing
$\alpha=\alpha_{1}+\alpha_{2}+\dotsm$ we have $\gamma(\alpha)=\gamma_{s}$ if and only if
$\alpha_{1}=\dotsm=\alpha_{s-1}=0$ and $\alpha_{s}\neq0$. Since the constant functions are harmonic
we actually know that $\gamma_{1}=1$ and $E_{1}$ is spanned by $(1,1,\dotsc,1)$.  We let $P_{s}$ be
the orthogonal projection onto $E_{s}$.

The subspaces $E_{s}$ provide the natural decomposition of harmonic functions into their scaling
components at $x$.  However we cannot expect to directly mimic
Definition~\ref{distribsatapoint_defnofdsatjunctionpt} because the estimate
\eqref{distribsatapoint_defnofLyupunov} does not imply the existence of a renormalized limit of the
form
\begin{equation}\label{distribsatapoint_LyupunovscalingofLaplacian}
    \lim_{i\rightarrow\infty} \gamma_{s}^{-i} P_{s} H_{i} \bigl( u - G D^{k-1} \Delta u  \bigr)
    \end{equation}
Indeed it is easy to see that \eqref{distribsatapoint_defnofLyupunov} does not even imply that
$A_{[w]_{i}}\alpha$ is $O(\gamma(\alpha))^{i}$.

A natural way to proceed was introduced in \cite{MR1025071,MR1761365} and further treated in
\cite{pelanderteplyaev07}. Let $u$ be the function we are considering, and $H_{i}u$ be the harmonic
function on $F_{[w]_{i}}(X)$ with boundary values equal to $u$ on $\partial F_{[w]_{i}}(X)$ as
usual. If we assume that the harmonic scaling matrices $A_{j}$ are all invertible we can unravel
the scaling structure for harmonic functions at $x$ by applying the inverse of $A_{[w]_{i}}$ to
$H_{i}u$.  For later use we record an elementary result about the scaling of the adjoint of
$A_{[w]_{i}}^{-1}$.

\begin{lemma}\label{distribsatpoint_scalingofAinverse}
If $\alpha\in E_{s}$ then $\lim \frac{1}{i}\log \bigl\| (A_{[w]_{i}}^{-1})^{\ast}
\alpha\bigr\|=-\log \gamma_{s}$.
\end{lemma}
\begin{proof}
Writing $\langle\cdot,\cdot\rangle$ for the usual inner product,
\begin{align*}
     \bigl\| (A_{[w]_{i}}^{-1})^{\ast} \alpha\bigr\|
    &= \sup_{\alpha''} \frac{ \bigl| \langle \alpha'', (A_{[w]_{i}}^{-1})^{\ast} \alpha \rangle \bigr|}{\|\alpha''\|}\\
    &= \sup_{\alpha'} \frac{ \bigl| \langle A_{[w]_{i}}\alpha', (A_{[w]_{i}}^{-1})^{\ast} \alpha \rangle
    \bigr|}{\|A_{[w]_{i}}\alpha'\|}\\
    &= \sup_{\alpha'} \frac{ \bigl| \langle \alpha', \alpha \rangle \bigr|}{\|A_{[w]_{i}}\alpha'\|}.
    \end{align*}
Since the logarithm is monotone, this implies
\begin{align*}
     \frac{1}{i} \log \bigl\| (A_{[w]_{i}}^{-1})^{\ast} \alpha\bigr\|
     &= \sup_{\alpha'}
     \biggl( \frac{1}{i}  \log \frac{ \bigl| \langle \alpha', \alpha \rangle \bigr|}{\|\alpha'\|} -
     \frac{1}{i}  \log \frac{\|A_{[w]_{i}}\alpha'\|}{\|\alpha'\|} \biggr)
     \end{align*}
however we know that the second term inside the supremum converges to $-\gamma(\alpha')$, whereas
the first converges to zero provided $\langle \alpha', \alpha \rangle\neq 0$. The latter condition
and $\alpha\in E_{s}$ requires that $\alpha'$ have a non-zero component in $E_{s}$, from which we
deduce $\gamma(\alpha')\geq \gamma_{s}$, with equality provided $P_{t}\alpha'=0$ for each $t<s$.
Combining these observations it is easy to see that for each $i$ the supremum is between
$-\gamma_{s} - \frac{c}{i}$ and $-\gamma_{s}$ for a constant $c$ independent of $i$.  It follows
that the limit in the statement of the lemma exists and has the asserted value.
\end{proof}

In order to account for the scaling behavior of the Laplacian, we set
\begin{equation}\label{distribsatapoint_asymptoticLaplacianscaling}
    \log \beta_{w}=\lim_{i\rightarrow\infty} \frac{1}{i} \log r_{[w]_{i}}\mu_{[w]_{i}}
    \end{equation}
provided the limit exists.

\begin{definition}\label{distribsatapoint_defnofdsatgenericpt}
Assume that $x=F_{w}(X)$ is a point at which the limits in~\eqref{distribsatapoint_defnofLyupunov}
and~\eqref{distribsatapoint_asymptoticLaplacianscaling} exist, and that all $A_{j}$ are invertible.
Inductively define derivatives $d_{s}$ and differentials $D^{k}$, $k\geq1$, at the point $x$ by
setting $D^{0}u=u(q)$, and for each $s$ such that $\beta_{w}^{k}< \gamma_{s} \leq \beta_{w}^{k-1}$
\begin{equation}\label{distribsatapoint_genericdefnofderiv}
    d_{s}u = \lim_{i\rightarrow\infty}  P_{s} A_{[w]_{i}}^{-1} H_{i} \bigl( u - G D^{k-1} \Delta u  \bigr)
    \end{equation}
if these limits exist. Note that  $d_{s}$ always exists for harmonic functions because the sequence
is constant in this case. Provided the necessary $d_{s}u$ exist we then let
\begin{equation}\label{distribsatapoint_genericdefnofdifferential}
    D^{k}u=h+ G D^{k-1}\Delta u
    \end{equation}
where $h$ is the unique harmonic function on $X$ with $d_{s}h=d_{s}u$ for those $s$ with
$\gamma_{s}>\beta_{w}^{k}$ and $d_{s}h=0$ for all other $s$.  We will also make use of
$\bar{D}^{k}u$, where $\bar{D}^{0}=u(0)$ and
\begin{equation*}
    \bar{D}^{k}u=\bar{h}+ G \bar{D}^{k-1}\Delta u
    \end{equation*}
with $\bar{h}$ harmonic on $X$ with $d_{s}\bar{h}=d_{s}u$ for those $s$ with
$\gamma_{s}\geq\beta_{w}^{k}$ and $d_{s}\bar{h}=0$ for all other $s$.
\end{definition}

Observe that this generalizes Definition~\ref{distribsatapoint_defnofdsatjunctionpt}, because if
$x=F_{w}(X)$ is a junction point then $w$ ends with infinite repetition of a single letter $j$, the
Lyapunov exponents are the eigenvalues of $A_{j}$, and the action of $A_{[w]_{i}}^{-1}$ on the
eigenspace $E_{s}$ is just multiplication by $\gamma_{s}^{-i}$.

The following result may be seen as a generalization of Theorem 1 of \cite{MR1761365}, see also
Theorems~5 and~6 of \cite{pelanderteplyaev07}.  It is proved by essentially the same method as
Lemma~\ref{distribsatapoint_existenceofdsatjunctionpoints}. At several points in the proof we use
the observation that for a positive sequence $a_{i}$ satisfying $\lim i^{-1}\log a_{i} =\log a$ and
a value $\epsilon>0$ there is a constant $C(\epsilon)$ so $C(\epsilon)^{-1}e^{-\epsilon i}
a^{i}\leq a_{i}\leq C(\epsilon) e^{\epsilon i} a^{i}$.

\begin{lemma}\label{distribsatapoint_harmonictangenetexistsatgeneric}
Assume that all $A_{j}$ are invertible, and that $x=F_{w}(X)$ is a point at which the limits
in~\eqref{distribsatapoint_defnofLyupunov} and~\eqref{distribsatapoint_asymptoticLaplacianscaling}
exist. For $u\in \dom(\Delta^{k})$ and each $s$ such that $\gamma_{s}>\beta_{w}^{k}$, the
derivative $d_{s}$ exists, and
\begin{equation}\label{distribsatapoint_genericesttoprovedsuisorderkdistrib}
    |d_{s}u |\leq C(k)\sum_{l=0}^{k} \|\Delta^{l} u\|_{\infty}.
    \end{equation}
For all sufficiently small $\epsilon>0$, the differential satisfies
\begin{equation}\label{distribsatapoint_genericdifferentialestimate}
    \bigl\| u- D^{k}u \bigr\|_{L^{\infty}(F_{[w]_{i}}(X))}
    \leq  C(k,\epsilon) \beta_{w}^{ik} e^{\epsilon i} \|\Delta^{k} u\|_{\infty}.
    \end{equation}
If in addition we assume that $\Delta^{k}u\in\dom(E)$ then
\begin{equation}\label{distribsatapoint_genericbardiffestimate}
    \bigl\| u- \bar{D}^{k}u \bigr\|_{L^{\infty}(F_{[w]_{i}}(X))}
    \leq  C(k,\epsilon) r_{[w]_{i}}^{1/2} \beta_{w}^{ik} e^{\epsilon i} E^{1/2}\bigl( \Delta^{k} u\bigr).
    \end{equation}
\end{lemma}

\begin{proof}
The proof is inductive.  When $k=0$ there are no $s$ with $\gamma_{s}>1=\beta_{w}^{0}$, so the
first statement is vacuous and \eqref{distribsatapoint_genericdifferentialestimate} is immediate.
Suppose both hold up to $k-1$.

Write $u-GD^{k-1}\Delta u$ as $H_{0}u+G(\Delta u - D^{k-1}\Delta u)$, so
\begin{equation}\label{distribsatapoint_genericdecompofdsuviadifferential}
    d_{s}u
    =d_{s}H_{0}u + \lim_{i} P_{s} A_{[w]_{i}}^{-1} H_{i} \bigl( G\Delta u - G D^{k-1}\Delta u \bigr)
    \end{equation}
provided the latter limit exists.   Writing $G_{i}$ for the Dirichlet Green's operator on the cell
$F_{[w]_{i}}(X)$, we have on that cell
\begin{equation*}
    G (\Delta u- D^{k-1}\Delta u )
    = H_{i} G (\Delta u- D^{k-1}\Delta u ) + G_{i} (\Delta u- D^{k-1}\Delta u )
    \end{equation*}
from which
\begin{equation*}
    H_{i+1} G (\Delta u- D^{k-1}\Delta u )
    = A_{w_{i+1}} H_{i} G (\Delta u- D^{k-1}\Delta u ) + H_{i+1} G_{i} (\Delta u- D^{k-1}\Delta u),
    \end{equation*}
therefore
\begin{multline}
     A_{[w]_{i+1}}^{-1} H_{i+1} \bigl( G\Delta u - G D^{k-1}\Delta u \bigr)
         -  A_{[w]_{i}}^{-1} H_{i} \bigl( G\Delta u - G D^{k-1}\Delta u \bigr) \\
    =  A_{[w]_{i}}^{-1} H_{i+1} G_{i} (\Delta u- D^{k-1}\Delta u),
    \label{distribsatapoint_estimateforcauchybehavioroftailsinderivativelemma}
    \end{multline}
and by substitution into~\eqref{distribsatapoint_genericdecompofdsuviadifferential},
\begin{equation}\label{distribsatapoint_genericdsuasseries}
    d_{s}u
    =d_{s}H_{0}u + \sum_{0}^{\infty} P_{s} A_{[w]_{i}}^{-1} H_{i+1} G_{i} (\Delta u- D^{k-1}\Delta u)
    \end{equation}
provided that the series converges.

Since $G_{i}$ inverts the Laplacian with Dirichlet boundary conditions on $F_{[w]_{i}}(X)$, we have
for any sufficiently small $\epsilon>0$ the bound
\begin{align}
    \Bigl| G_{i} (\Delta u- D^{k-1}\Delta u ) \Bigr|
    &\leq C r_{[w]_{i}} \mu_{[w]_{i}} \bigl\| \Delta u -D^{k-1}\Delta u
    \bigr\|_{L^\infty(F_{[w]_{i}(X)})} \notag\\
    &\leq C(k-1,\epsilon) r_{[w]_{i}} \mu_{[w]_{i}} \beta_{w}^{(k-1)i} e^{(\epsilon/4) i}
        \bigl\| \Delta^{k} u \bigr\|_{\infty} \notag\\
    &\leq C(k-1,\epsilon) \beta_{w}^{ki} e^{(\epsilon/2) i} \bigl\| \Delta^{k} u \bigr\|_{\infty}
    \label{distribsatapoint_genericestimateforGiappliedtouminusdifferential}
    \end{align}
because of the inductive hypothesis~\eqref{distribsatapoint_genericdifferentialestimate} and the
Laplacian scaling estimate~\eqref{distribsatapoint_asymptoticLaplacianscaling}. This is also
applicable to $H_{i+1}G_{i} (\Delta u- D^{k-1}\Delta u )$ by the maximum principle. Using
Lemma~\ref{distribsatpoint_scalingofAinverse} to estimate the size of
$\|(A_{[w]_{i}}^{-1}\bigr)^{*}P_{s}\alpha\|$, it follows that for any sufficiently small
$\epsilon>0$, and any vector $\alpha$,
\begin{align*}
    \Bigl| \langle P_{s} A_{[w]_{i}}^{-1} H_{i+1} G_{i} (\Delta u- D^{k-1}\Delta u), \alpha \rangle
    \Bigr|
    &= \Bigl| \langle  H_{i+1} G_{i} (\Delta u- D^{k-1}\Delta u),
     \bigl(A_{[w]_{i}}^{-1}\bigr)^{*}P_{s}\alpha \rangle \Bigr| \\
    &\leq  C(k-1,\epsilon) \beta_{w}^{ki} \gamma_{s}^{-i} e^{(3\epsilon/4) i}
     \bigl\| \Delta u \bigr\|_{\infty},
    \end{align*}
This and the assumption $\gamma_{s}>\beta_{w}^{k}$ imply that if $\epsilon>0$ was chosen small
enough then the series in~\eqref{distribsatapoint_genericdsuasseries} converges, and is bounded by
$C\|\Delta^{k}u\|_{\infty}$.  The
estimate~\eqref{distribsatapoint_genericesttoprovedsuisorderkdistrib} follows because $d_{s}H_{0}u$
is bounded by $C\|u\|_{\infty}$.

Now $u-D^{k}u = u-h-GD^{k-1}\Delta u = H_{0}u - h + G(\Delta u - D^{k-1}\Delta u)$, where $h$ is
the harmonic function with $d_{s}u=P_{s}h$ for all $s$ satisfying $\gamma_{s}>\beta_{w}^{k}$ and
$P_{s}h=0$ otherwise.  An expression for $h$ can be obtained by
summing~\eqref{distribsatapoint_genericdsuasseries} over these values of $s$.  Comparing it to the
expression
\begin{equation*}
    A_{[w]_{i}}^{-1} H_{i} (u-D^{k}u)
    = H_{0}u - h + \sum_{0}^{i-1} A_{[w]_{l}}^{-1} H_{l+1} G_{l} (\Delta u - D^{k-1}\Delta u )
    \end{equation*}
from~\eqref{distribsatapoint_estimateforcauchybehavioroftailsinderivativelemma}, it is apparent
that for those $s$ with $\gamma_{s}>\beta_{w}^{k}$ we have
\begin{equation*}
    P_{s} A_{[w]_{i}}^{-1} H_{i}(u-D^{k}u)
    = - \sum_{i}^{\infty} P_{s} A_{[w]_{l}}^{-1} H_{l+1} G_{l} (\Delta u - D^{k-1} \Delta u)
    \end{equation*}
which we note satisfies for all $\|\alpha\|\leq 1$ and sufficiently small $\epsilon>0$
\begin{align}
    \bigl| \langle P_{s} A_{[w]_{i}}^{-1} H_{i}(u-D^{k}u), \alpha \rangle\bigr|
    &\leq \sum_{i}^{\infty} \bigl| \langle H_{l+1} G_{l} (\Delta u - D^{k-1} \Delta u),
    (A_{[w]_{l}}^{-1})^{\ast} P_{s} \alpha \rangle\bigr| \notag\\
    &\leq \sum_{i}^{\infty}
    C(k-1,\epsilon) \beta_{w}^{kl} \gamma_{s}^{-l} e^{(3\epsilon/4) l}
    \bigl\| \Delta^{k} u \bigr\|_{\infty} \notag\\
    &\leq C(k-1, \epsilon) \beta_{w}^{ki} \gamma_{s}^{-i} e^{(3\epsilon/4) i}
    \bigl\| \Delta^{k} u \bigr\|_{\infty}.
    \label{distribsatapoint_genericintermediateestonefordslemma}
    \end{align}
For those $s$ satisfying $\gamma_{s}\leq\beta_{w}^{k}$ we have instead
\begin{equation*}
    P_{s} A_{[w]_{i}}^{-1} H_{i}(u-D^{k}u)
    = \sum_{0}^{i} P_{s} A_{[w]_{l}}^{-1} H_{l+1} G_{l} (\Delta u - D^{k-1} \Delta u).
    \end{equation*}
and for all vectors $\alpha$ with $\|\alpha\|\leq 1$,
\begin{align}
    \bigl| \langle P_{s} A_{[w]_{i}}^{-1} H_{i}(u-D^{k}u), \alpha \rangle\bigr|
    &\leq \sum_{0}^{i}
    C(k-1,\epsilon) \beta_{w}^{kl} \gamma_{s}^{-l} e^{(3\epsilon/4) l}
    \bigl\| \Delta^{k} u \bigr\|_{\infty} \notag\\
    &\leq C(k-1,\epsilon) \beta_{w}^{ki} \gamma_{s}^{-i} e^{(3\epsilon/4) i}
    \bigl\| \Delta^{k} u \bigr\|_{\infty}.
    \label{distribsatapoint_genericintermediateesttwofordslemma}
    \end{align}
Equations~\eqref{distribsatapoint_genericintermediateestonefordslemma}
and~\eqref{distribsatapoint_genericintermediateesttwofordslemma} give the same estimate for each
$P_{s} A_{[w]_{i}}^{-1} H_{i}(u-D^{k}u)$.  Mapping forward again by $A_{[w]_{i}}$ increases each
term by a factor at most $C(\epsilon)\gamma_{s}^{i}e^{(\epsilon/4) i}$, so summing over all $s$ we
finally have
\begin{equation*}
    \bigl| H_{i}(u-D^{k}u) \bigr|
    \leq C \beta_{w}^{ki} e^{\epsilon i} \bigl\| \Delta^{k} u \bigr\|_{\infty}
    \end{equation*}
for some constant $C=C(k,\epsilon)$.  Now the restriction of $(u-D^{k}u)$ to $F_{[w]_{i}}(X)$ is
\begin{equation*}
    (u-D^{k}u) \evald{F{[w]_{i}}(X)}
    = H_{i}(u-D^{k}u) + G_{i}\bigl( \Delta (u- D^{k}u) \bigr)
    = H_{i}(u-D^{k}u) + G_{i}\bigl( \Delta u- D^{k-1}\Delta u) \bigr)
    \end{equation*}
the second term of which is bounded by $\beta_{w}^{ki}e^{\epsilon i} \bigl\| \Delta^{k} u
\bigr\|_{\infty}$ from~\eqref{distribsatapoint_genericestimateforGiappliedtouminusdifferential},
and the first term of which we have just estimated in the same way.  This
establishes~\eqref{distribsatapoint_genericdifferentialestimate} and completes the induction.

The proof of~\eqref{distribsatapoint_genericbardiffestimate} is the same, except
that~\eqref{distribsatapoint_genericbardiffestimate} is used in place
of~\eqref{distribsatapoint_genericdifferentialestimate} throughout.  The validity of the estimate
for $k=0$ is a consequence of the H\"{o}lder estimate~\eqref{setting_resistmetricestimate}.
\end{proof}

As previously, we fix orthonormal bases for the spaces $E_{s}$ and see that the components of
$d_{s}$ are distributions.

\begin{corollary}\label{distribsatapoint_genericordersofdistribs}
Suppose that $x$ satisfies the assumptions of
Lemma~\ref{distribsatapoint_harmonictangenetexistsatgeneric} and $\beta_{w}^{k}<\gamma_{s}$. Any
component $d_{s,v}$ of the derivative $d_{s}$ is a distribution of order at most $k$ supported at
$x$.  If also $\gamma_{s}<\beta_{w}^{k-1}$ then $d_{s,v}$ has order equal to $k$. If $d_{s,v}$ has
order $k$ then defining $\Delta^{l}d_{s,v}$ by $\Delta^{l}d_{s,v}\testfn =
d_{s,v}\Delta^{l}\testfn$ yields a distribution supported at $x$ and of order $k+l$.
\end{corollary}
\begin{proof}
Linearity of $d_{s,v}$ is immediate from Definition~\ref{distribsatapoint_defnofdsatgenericpt}, so
it is a distribution of order at most $k$
by~\eqref{distribsatapoint_genericesttoprovedsuisorderkdistrib}.  Again using
Definition~\ref{distribsatapoint_defnofdsatgenericpt} it is apparent that $d_{s,v}\testfn=0$ if
$\testfn\in\testfns$ vanishes in a neighborhood of $x$, so $d_{s,v}$ is supported at $x$.

To see that $d_{s,v}$ has order at least $k$, consider the harmonic function $h$ with boundary
values equal to the unit vector in the $v$ direction in $E_{s}$.  Then $H_{i}h=A_{[w]_{i}}H_{0}h$,
so the sequence in~\eqref{distribsatapoint_genericdefnofderiv} is constant equal to $H_{0}h$, and
$d_{s,v}h=1$.  Now for $\epsilon>0$ so small that $\gamma_{s}e^{3\epsilon}\leq \beta_{w}^{k-1}$ we
have
\begin{equation*}
    \bigl\| H_{i}h \bigr\|_{\infty}
    \leq C(\epsilon) \gamma_{s}^{i} e^{\epsilon i}
    \leq C(\epsilon) \beta_{w}^{(k-1)i} e^{-2\epsilon i}
    \leq C(\epsilon) \big( r_{[w]_{i}}\mu_{[w]_{i}} \bigr)^{k-1} e^{-\epsilon i}
    =o\big( r_{[w]_{i}}\mu_{[w]_{i}} \bigr)^{k-1}
    \end{equation*}
and of course $\Delta^{l}h\equiv0$ for all $l>0$, so
Theorem~\ref{distribsatapoint_decayofHijtestfnimpliesgoodextension} applies with $m=k-1$, and there
is a test function $\testfn$ such that $\testfn=h$ in a neighborhood of $x$ and $|\testfn|_{k-1}$
is as small as we desire.  Since $d_{s,v}h=d_{s,v}\testfn$ by the support condition, $d_{s,v}$
cannot be of order $k-1$ or less. The final statement of the lemma is obvious.
\end{proof}

\begin{theorem}\label{distribsatapoint_identifydistribsatgenericpoints}
Suppose that all of the matrices $A_{j}$ are invertible, and that $x=F_{w}(X)$ is a point at which
the limits in~\eqref{distribsatapoint_defnofLyupunov}
and~\eqref{distribsatapoint_asymptoticLaplacianscaling} exist.  Then all distributions of order at
most $k$ at $x$ are linear combinations of the distributions $\Delta^{l}d_{s,v}$, with
$\gamma_{s}\geq\beta_{w}^{k-l}$.
\end{theorem}

\begin{proof}
As in the proof of Theorem~\ref{distribsatapoint_identifyingdistribsatjunctionpoint}, it suffices
to show that $T$ vanishes whenever the distributions $\Delta^{l}d_{s,v}$, with
$\gamma_{s}\geq\beta_{w}^{k-l}$ vanish.

Suppose $\testfn\in\testfns$ satisfies $\Delta^{l}d_{s,v}\testfn=0$ for those
$\gamma_{s}\geq\beta_{w}^{k-l}$. Then the differential $\bar{D}^{k}\testfn$ (which exists by
Lemma~\ref{distribsatapoint_harmonictangenetexistsatgeneric}) must be zero, as must
$\bar{D}^{k-l}\Delta^{l}\testfn$ for each $0\leq l\leq k$.
From~\eqref{distribsatapoint_genericbardiffestimate} we then see that for all sufficiently small
$\epsilon>0$,
\begin{align*}
    \bigl\| H_{i} \Delta^{l}\testfn \bigr\|_{L^{\infty}(F_{[w]_{i}}(X))}
    \leq \bigl\| \Delta^{l}\testfn \bigr\|_{L^{\infty}(F_{[w]_{i}}(X))}
    &\leq  C(k,\epsilon) r_{[w]_{i}}^{1/2} \beta_{w}^{i(k-l)} e^{\epsilon i} E^{1/2}\bigl( \Delta^{k}
    u\bigr)\\
    &\leq  C(k,\epsilon) r_{[w]_{i}}^{1/2} (r_{[w]_{i}} \mu_{[w]_{i}})^{k-l} e^{2\epsilon i} E^{1/2}\bigl( \Delta^{k}
    u\bigr)\\
    &=o(r_{[w]_{i}} \mu_{[w]_{i}})^{k-l}.
    \end{align*}

Applying Theorem~\ref{distribsatapoint_decayofHijtestfnimpliesgoodextension} we find that for any
$\delta>0$ there is $\psi$ equal to $\testfn$ in a neighborhood of $x$ and such that
$|\psi|_{k}<\delta$. In particular, since $T$ is order $k$ and supported at $x$, there is $M$
independent of $\testfn$ such that
\begin{equation*}
    |T\testfn|=|T\psi| \leq M|\psi| \leq \delta
    \end{equation*}
and thus $T\testfn=0$.
\end{proof}

In concluding this section it seems appropriate to say something about the set of points $x$
satisfying the conditions in Definition~\ref{distribsatapoint_defnofdsatgenericpt}.  The set at
which the limit $\beta_{w}$ exists has full $\mu$-measure by the law of large numbers, and in fact
\begin{equation*}
    \beta_{w}= \sum_{j}\mu_{j}\log r_{j}\mu_{j}
    \end{equation*}
at $\mu$-a.e. point.  The set on which the Lyupunov exponents exist may be treated by the theory of
random matrices introduced by Furstenberg  and Kesten \cite{MR0121828}.  In particular, it is
possible to make certain assumptions on the matrices $A_{i}$ that guarantee that this set is also
of full $\mu$-measure. This topic is discussed quite thoroughly in the paper
\cite{pelanderteplyaev07} of Pelander and Teplyaev, so we will not cover it here.  One consequence
of their work, however, is that there are conditions that imply the spaces $E_{s}$ are independent
of the choice of point $x$. For example, if the semigroup generated by the $A_{i}$ is strongly
irreducible and contracting then there is a single vector $\alpha_{1}$ such that at $\mu$-almost
every $x$, the space $E_{1}$ is spanned by $\alpha_{1}$ and has scaling $\gamma_{1}$. If the same
strong irreducibility and contraction holds after taking the quotient to remove $E_{1}$ then
$E_{2}$ is also one-dimensional and independent of $x$ on a full measure set.  For a fractal where
the irreducibility and contraction properties are true for the semigroup generated by the $A_{i}$
on each of the subspaces found by removing $E_{1}, E_{2},..,E_{s-1}$ in turn, we could conclude
that all of the distributions of the form $d_{s}$ are independent of $x$ on a set of full
$\mu$-measure.  Hence in this situation any distribution of order $m$ with point support in a fixed
set of full $\mu$-measure would be a finite linear combination of distributions $\Delta^{l}d_{s}$
for suitable values of $l$, where the $d_{s}$ are independent of $x$.  This generic behavior is
very different from that seen at junction points and eventually periodic points, where the
structure of point-supported distributions can vary substantially from point to point.

\section{Distributions on products}

In this section we give a theory of distributions on finite products of post-critically finite
self-similar fractals, using the analytic theory for such products developed in~\cite{MR2095624}.
This gives genuinely new examples, because products of p.c.f. self-similar sets are not usually
themselves p.c.f. Since there is no essential difference between a product $X=X'\times X''$ with
two factors and a general finite product, we state our results only for the two factor case.

Following the notational conventions of~\cite{MR2095624}, points are $x=(x',x'')$, functions on $X$
are called $u$ or $f$, on $X'$ they are $u'$ or $f'$, while on $X''$ they are $u''$ or $f''$. The
energies on $X'$ and $X''$ are $\DF'$ and $\DF''$ and the Laplacians are $\Delta'$ and $\Delta''$.
They come from a regular harmonic structure as in Section~\ref{setting_section} and have energy and
measure scaling factors $r'$, $\mu'$, $r''$ and $\mu''$.  The corresponding Laplacians $\Delta'$
and $\Delta''$ are defined componentwise, so $u\in\dom(\Delta')$ with $\Delta' u=f$ if $u$ and $f$
are continuous on $X$ and have the property that for each fixed $x''\in X''$ we have
$\Delta'u(\cdot,x'')=f(\cdot,x'')$.  A similar definition is used for $\Delta''$. By Lemma~11.2
of~\cite{MR2095624}, $\Delta'$ and $\Delta''$ commute on $\dom(\Delta')\cap\dom(\Delta'')$.

\begin{definition}
A function $u$ on $X$ is smooth if for all $j,k\in\mathbb{N}$, $(\Delta')^{j}(\Delta'')^{k}u$ is a
continuous function on $X$.  The definition extends to a finite union of cells in the obvious
manner, and $u$ is smooth on a domain in $X$ if it is smooth on every finite union of cells in the
domain.
\end{definition}

We define the test functions on a domain $\domain$ to be the smooth functions of compact support
with the usual topology and the corresponding seminorms
\begin{equation*}
    |\testfn|_{m}
    = \sup \Bigl\{ \bigl|(\Delta')^{j}(\Delta'')^{k}\testfn(x)\bigr|: x\in\domain,\, j+k\leq m''
    \Bigr\}.
    \end{equation*}
The distributions form the dual space with weak-star topology.  A distribution $T$ has order $m$ if
on any compact $K$ there is $M=M(K)$ so that $|T\testfn|\leq M|\testfn|_{m}$ for all test functions
supported on $K$

The goal of this section is to provide conditions under which analogues of our main results for
distributions on p.c.f. fractals are also valid on products of these fractals. In order to avoid
duplicating a great deal of work, we only give details of the proofs where they differ
significantly from those for the case of a single p.c.f. fractal.  In particular it is fairly easy
to verify that all of the results of Section~\ref{testfunction_section} (except
Corollary~\ref{IntersectofSobolevspaces}), Section~\ref{distributions_section}, and
Section~\ref{structuresection} prior to
Theorem~\ref{structure_positivedistribsarepositivemeasures}, depend only on the partitioning
property of Theorem~\ref{setting_partitionthm} and the estimate~\eqref{setting_extensionestimate}
(either directly or through Lemma~\ref{testfunction_partition}) as well as the fact that for
compact $\Omega$ there is an orthonormal basis of $L^2$ consisting of eigenfunctions.  The latter
is obviously true for the product $X'\times X''$ because it is true for the factors, so the
original proofs transfer to the product setting once we know the partitioning property and the
corresponding estimate for product spaces.  These are proved in Theorem~\ref{productparitioning}
and~\eqref{productpartitioningestimate} below, so all of the aforementioned results are also true
for products of p.c.f. fractals with regular harmonic structure, and connected fractafolds with
restricted cellular structure based on such products.

Only small changes are needed to obtain analogues of the remaining results from
Section~\ref{structuresection}.  The proof of
Theorem~\ref{structure_positivedistribsarepositivemeasures} required that on any cell there was a
positive smooth function equal to $1$ on the cell and vanishing outside a specified neighborhood:
such a function may be obtained on the product space as a product of functions of this type on the
factors, so the theorem is true for products in which each factor has the
estimate~\eqref{setting-subGaussianbound} for the heat kernel.  The other results are used to prove
the structure theorem (Theorem~\ref{structure_structuretheorem}).  Of these,
Lemma~\ref{structure_greensfnforKexists} remains true with the same proof if it is modified to say
that $\Delta'$ maps $\testfnson(K)$ to itself with image orthogonal to those $\testfn$ having
$\Delta'\testfn=0$ and there is $\tilde{G}'_{K}$  such that
$-\Delta'\tilde{G}'_{K}(\Delta'\testfn)=\Delta'\testfn$; there is a corresponding result for
$\Delta''$.  A version of Theorem~\ref{structure_Gondistribus} is then true with
$\tilde{G}''_{K}\tilde{G}'_{K}$ replacing $\tilde{G}_{K}$ throughout.  The original proof shows
that for $m\geq1$, $\tilde{G}''_{K}\tilde{G}'_{K}$ takes a distribution of order $m$ to one of
order at most $m-1$. To show that $\tilde{G}''_{K}\tilde{G}'_{K}$ takes a distribution of order
zero to a continuous function it suffices to approximate the corresponding measure $\nu$ by a
sequence of linear combinations of product measures. Applying $\tilde{G}''_{K}\tilde{G}'_{K}$ to a
product measure gives a continuous function by the original proof of
Theorem~\ref{structure_Gondistribus}, so applying it to the sequence gives a uniformly convergent
sequence of continuous functions whose limit represents the distribution
$\tilde{G}''_{K}\tilde{G}'_{K}\nu$.  The proof of Theorem~\ref{structure_structuretheorem} needs no
further changes.

At this point we have essentially all of the results of Sections~\ref{testfunction_section},
\ref{distributions_section}, and~\ref{structuresection} in the product setting (the only exception
is Corollary~\ref{IntersectofSobolevspaces}). In addition there are some things that can be said
about distributions with point support that generalize the results of
Section~\ref{distribsatpoint_section}.  We will return to these after giving the details of the
partitioning argument, because some aspects of the procedure for cutting off a smooth function will
be important for the proofs.

\subsection*{Partitioning on products}

We prove analogues of the partitioning property in Theorem~\ref{setting_partitionthm} and the
estimate~\eqref{setting_extensionestimate} in the product setting.  As in the single variable case,
the proof relies on a cell-by-cell construction of a smooth function, for which the following
matching condition is essential.  Note that a cell in $X$ is a product of cells from $X'$ and
$X''$, so has the form $K=F'_{w'}(X')\times F''_{w''}(X'')$, where $w'$ and $w''$ are finite words.
Its boundary consists of faces $\{q'_{i}\}\times F''_{w''}(X'')$ and $F'_{w'}(X')\times\{q''_{j}\}$
for $q'_{i}\in V'_{0}$ and $q''_{j}\in V''_{0}$.

\begin{lemma}\label{productmatchinglemma}
Suppose the cells $K_{1},\dotsc,K_{k}$ all contain the face $L=\{q'\}\times F''_{w''}(X'')$, and
that the union $\cup_{1}^{k} K_{l}$ contains a neighborhood of every point in $L$ except those of
the form $\bigl(q',F''_{w''}q''\bigr)$ with $q''\in V_{0}$.  If $u_{j}$ is smooth on $K_{j}$ for
each $j$, then the piecewise defined function $u=u_{j}$ on $K_{j}$ is smooth on $\cup_{1}^{k}
K_{l}$ if and only if for each $x''\in X''$, both\\
(a) The functions $(\Delta')^{l}(\Delta'')^{m}u_{j}(q',x'')$ are independent of $j$ for each $l$ and $m$, and\\
(b) For each $x''$, $\sum_{j}(\partial'_{n})_{j}(\Delta')^{l}(\Delta'')^{m}u_{j}(q',x'')=0$, where
$(\partial'_{n})_{j}$ indicates the normal derivative in the $x'$ variable from within $K_{j}$.
\end{lemma}

\begin{proof}
For fixed $x''$, (b) is the necessary and sufficient matching condition in the first variable for
$(\Delta')^{l}(\Delta'')^{m}u(\cdot,x'')$ to exist (as a function rather than a measure with atom
at $q'$). Condition (a) is then equivalent to continuity of $(\Delta')^{l}(\Delta'')^{m}u$.
\end{proof}

Our construction uses an analogue of the Borel theorem from~\cite{RST}.  That result yields the
existence of a smooth function with a prescribed jet at a junction point of a pcf fractal, whereas
we need existence of a smooth function with prescribed smooth jet on the face of a cell in the
product $X$.

\begin{theorem}\label{productBoreltheorem}
Fix a face $\{q'\}\times X''$ and a neighborhood $U\subset X'$ of $q'$.  Given two sequences
$\{\rho_{l}(x'')\}_{l=0}^{\infty}$ and $\{\sigma_{l}(x'')\}_{l=0}^{\infty}$ of functions that are
smooth in $x''$, there is a smooth function $u$ with support in $U\times X''$ such that for each
$x''\in X''$, $(\Delta')^{k}(\Delta'')^{m}u(q',x'')=(\Delta'')^{m}\rho_{k}(x'')$ and
$\partial'_{n}(\Delta')^{k}(\Delta'')^{m}u(q',x'')=(\Delta'')^{m}\sigma_{k}(x'')$.
\end{theorem}

\begin{proof}
The proof is almost the same as that for Theorem~4.3 of~\cite{RST}.  Specifically we form the
series
\begin{equation}\label{seriesdefnforBorelonproduct}
    u(x',x'')=\sum_{l} \rho_{l}(x'')g_{l,m_{l}}(x') + \sigma_{l}(x'') f_{l,n_{l}}(x')
    \end{equation}
where the functions $g_{l,m_{l}}$ and $f_{l,n_{l}}$ are as defined in that proof, so they satisfy
\begin{align*}
    (\Delta')^{k}g_{l,m_{l}}(q)&=\delta_{kl}&\partial'_{n}(\Delta')^{k}g_{l,m_{l}}(q)&=0\\
    (\Delta')^{k}f_{l,n_{l}}(q)&=0          &\partial'_{n}(\Delta')^{k}f_{l,n_{l}}(q)&=\delta_{kl}
    \end{align*}
and have supports in cells of scale $m_{l}$ and $n_{l}$ respectively.  Convergence of the
series~\eqref{seriesdefnforBorelonproduct} is achieved by making an appropriate choice of $m_{l}$
and $n_{l}$.  In particular, it follows from the cited proof that if $|\rho_{l}(x'')|\leq R_{l}$
and $|\sigma_{l}(x'')|\leq S_{l}$ for all $x''$, then one can choose $m_{l}$ and $n_{l}$ depending
only on $R_{l}$ and $S_{l}$ such that for each $x''$ the series converges to a function that is
smooth in $x'$, supported in $U\times X''$, and has $(\Delta')^{k}u(q,x'')=\rho_{k}(x'')$ and
$\partial'_{n}(\Delta')^{k}u(q,x'')=\sigma_{k}(x'')$.

Now we require convergence not only of the series for $u(x',x'')$, but also that for
$(\Delta'')^{m}u(x',x'')$ for each $m$, so we must diagonalize.  Set
\begin{align*}
    R_{l} &= \max_{0\leq l''\leq l} \max_{x''\in X''} \bigl| (\Delta'')^{l''}\rho_{l}(x'') \bigr|\\
    S_{l} &= \max_{0\leq l''\leq l} \max_{x''\in X''} \bigl| (\Delta'')^{l''}\sigma_{l}(x'') \bigr|
    \end{align*}
which are finite by the assumed smoothness and the compactness of $X''$, and let $m_{l}$ and
$n_{l}$ be chosen as described above. For fixed $x''$, all terms after the $m$-th in the partial
sum
\begin{equation*}
    (\Delta'')^{m} \sum_{l}^{L} \rho_{l}(x'')g_{l,m_{l}}(x') + \sigma_{l}(x'') f_{l,n_{l}}(x')
    \end{equation*}
have coefficients bounded by $R_{l}$ and $S_{l}$, so the above reasoning implies that the partial
sums converge to a function that is smooth in $x'$, and has
$(\Delta')^{k}(\Delta'')^{m}u(q',x'')=(\Delta'')^{m}\rho_{k}(x'')$ and
$\partial'_{n}(\Delta')^{k}(\Delta'')^{m}u(q',x'')=(\Delta'')^{m}\sigma_{k}(x'')$ for all $m$.

Finally, it will be useful later to have estimated the contribution of each term to the
$L^{\infty}$ norm of $(\Delta')^{k}(\Delta'')^{m}u$.  It is convenient to write $w'(m_{l})$ and
$w'(n_{l})$ for the words such that $F_{w'(m_{l})}(X')$ is the support of $g_{l,n_{l}}$ and
$F'_{w'(n_{l})}(X)$ is the support of $f_{l,n_{l}}$.  Note that scaling then implies (see
equations~4.4 and~4.5 of~\cite{RST}) that
\begin{align*}
    \bigl\|(\Delta')^{k} (\Delta'')^{m} \rho_{l}(x'')g_{l,m_{l}}(x')\bigr\|
    &\leq c(k,l) \bigl(r'_{w'(m_{l})}\mu'_{w'(m_{l})}\bigr)^{l-k} \bigl\|(\Delta'')^{m}
    \rho_{l}(x'')\bigr\|_{\infty}\\
    \bigl\|(\Delta')^{k} (\Delta'')^{m} \sigma_{l}(x'')f_{l,n_{l}}(x')\bigr\|
    &\leq c(k,l) \bigl(r'_{w'(n_{l})}\mu'_{w'(n_{l})}\bigr)^{l-k} r'_{w'(n_{l})} \bigl\|(\Delta'')^{m}
    \sigma_{l}(x'')\bigr\|_{\infty}
    \end{align*}
and in the construction in~\cite{RST} it is noted that the contributions of terms with $l>k$ may be
made smaller than any prescribed $\epsilon>0$, so taking $\epsilon$ to be a small multiple of
$\|\rho_{0}\|_{\infty}$ we obtain
\begin{align}\label{Boreltheoremestimateforderivs}
    \bigl\| (\Delta')^{k} (\Delta'')^{m} u \bigr\|_{\infty}
    &\leq \sum_{l=0}^{k} c(k,l) \bigl(r'_{w'(m_{l})}\mu'_{w'(m_{l})}\bigr)^{l-k}
        \bigl\|(\Delta'')^{m} \rho_{l}(x'')\bigr\|_{\infty} \notag\\
    &\quad+ \sum_{l=0}^{k-1} c(k,l) \bigl(r'_{w'(n_{l})}\mu'_{w'(n_{l})}\bigr)^{l-k}
        r'_{w'(n_{l})} \bigl\|(\Delta'')^{m} \sigma_{l}(x'')\bigr\|_{\infty}.
    \end{align}
\end{proof}

\begin{remark}
This result may be localized to any cell in $X$ simply by rescaling the desired jet for the cell to
obtain a corresponding jet on $X$, applying the theorem, and then composing the resulting function
with the inverse of the map to the cell.  It may also be applied to a face in a finite union of
cells, so that the face is of the form $\{q'\}\times(\cup_{j=1}^{J}K''_{j})$ with each $K''_{j}$ a
cell in $X''$, because $\cup_{j=1}^{J}K''_{j}$ is compact.
\end{remark}

In order to make use of the preceding result we require a small lemma.
\begin{lemma}\label{traceofnormalderivativeissmooth}
If $u$ is smooth on $X$ and $q'\in V'_{0}$ then $\partial'_{n}u(q',x'')$ is smooth with respect to
$x''$ and $(\Delta'')^{l}\partial'_{n}u(q',x'')=\partial'_{n}(\Delta'')^{l}u(q',x'')$.  There is a
bound
\begin{equation}\label{estimatefornormalderivativesinproductusinghigherLap}
    \bigl\|\partial'_{n}(\Delta'')^{l}u(q',x'')\|_{\infty}
    \leq C\Bigl(\bigl\|(\Delta'')^{l}u\bigr\|_{\infty} + \bigl\|\Delta'(\Delta'')^{l}u\bigr\|_{\infty}\Bigr)
    \end{equation}
\end{lemma}
\begin{proof}
For each $x''$ and each scale $m$, let $h_{m}(x',x'')$ be the function that is piecewise harmonic
at scale $m$ in the $x'$ variable and coincides with $u$ on $V'_{m}\times\{x''\}$.  Then
$h_{m}(x',x'')$ is smooth in $x''$, because its values are obtained as uniform limits of linear
combinations of the values from $V'_{m}\times\{x''\}$.  Moreover, the normal derivative
$\partial'_{n}h_{m}(q',x'')$ is a linear combination (with coefficients depending on $m$) of the
differences $\bigl(h_{m}(p_{1}',x'')-h_{m}(p_{2}',x'')\bigr)$, where $p'_{1}$ and $p'_{2}$ are
neighbors of $q'$ at scale $m$.  Thus $\partial'_{n}h_{m}(q',x'')$ is smooth in $x''$ and
$(\Delta'')^{l}\partial'_{n}h_{m}(q',x'')=\partial'_{n}(\Delta'')^{l}h_{m}(q',x'')$.

For each fixed $x''$, we may express $(\Delta'')^{l}u(x',x'')$ on a cell $K'_{m}$ of scale $m$
containing $q'$ as the sum of $(\Delta'')^{l}h_{m}$ and an integral involving the Dirichlet Green
kernel $G'_{m}$ for $\Delta'$ on $K'_{m}$.  Taking the normal derivative we obtain
\begin{equation}\label{togetestimatefornormalderivonproductusinghigherlap}
    \partial'_{n} (\Delta'')^{l} u(q',x'')
    = \partial'_{n} (\Delta'')^{l} h_{m}(q',x'')
        + \int \bigl(\Delta' (\Delta'')^{l} u(y',x'')\bigr) \partial'_{n}G'_{m}(q',y')\,d\mu'(y').
    \end{equation}
However an easy scaling argument shows that $\partial'_{n}G'_{m}(q',y')$ is bounded independent of
$m$ and $y'$, so the integral term is bounded by a constant multiple of $\bigl\|\Delta'
(\Delta'')^{l} u\bigr\|_{\infty}\mu'(K'_{m})$, independent of $m$ and $x''$.  Since
$\mu'(K'_{m})\rightarrow0$ as $m\rightarrow\infty$ we conclude that
$(\Delta'')^{l}\partial'_{n}h_{m}(q',x'')$ converges to $\partial'_{n} (\Delta'')^{l} u(q',x'')$
uniformly in $x''$ for each $l$. Then~\eqref{estimatefornormalderivativesinproductusinghigherLap}
is obtained by using~\eqref{togetestimatefornormalderivonproductusinghigherlap} with $m=0$.
\end{proof}

We may use the preceding results to smoothly cut off a smooth function on a neighborhood of a cell.

\begin{theorem}\label{cutoffonproductcell}
Let $u$ be smooth on a cell $K=F'_{w'}(X')\times F''_{w''}(X'')$, and $U\supset K$ be open.  There
is a function $v$ such that $v=u$ on $K$, $v=0$ on $X\setminus U$ and $v$ is smooth on $X$.
Moreover for each $k$,
\begin{equation}\label{cutoffestimateonproduct}
    \bigl\|(\Delta')^{k}(\Delta'')^{m} v \bigr\|_{\infty}
    \leq C(k,U) \sum_{l=0}^{k} \sum_{n=0}^{m} \bigl\|(\Delta')^{l}(\Delta'')^{n}u \bigr\|_{L^{\infty}(K)}.
    \end{equation}
\end{theorem}

\begin{proof}
Let $K'=F'_{w'}(X')$ and $K''=F''_{w''}(X'')$.  Fix a face of $K$ having the form $\{q'\}\times
K''$ and let $\rho_{k}(x'')=(\Delta')^{k}u(q',x'')$ and
$\sigma_{k}(x'')=\partial'_{n}(\Delta')^{k}u(q',x'')$. The functions $\rho_{k}$ are smooth in $x''$
by the definition of smoothness of $u$, and the functions $\sigma_{k}$ are smooth in $x''$ by
Lemma~\ref{traceofnormalderivativeissmooth}.   Now take a finite number of small cells $K'_{j}$ in
$X'$ with the following properties: the intersection $K'\cap K'_{j}=\{q'\}$ for all $j$, the
intersection $K'_{j}\cap K'_{\tilde{j}}=\{q'\}$ for all $j\neq \tilde{j}$, the union
$K'\cup(\cup_{j}K'_{j})$ contains a neighborhood of $q'$ in $X'$, and
$\bigl(K'\cup(\cup_{j}K'_{j})\bigr)\times K''\subset U$. Let the number of $K'_{j}$ be $J$, and
apply Theorem~\ref{productBoreltheorem} to each $K'_{j}$ to obtain a smooth function $u_{j}$ that
has jets $\rho_{k}(x'')$ and $(-1/J)\sigma_{k}(x'')$ at $q'$ and is supported in a neighborhood of
$q'$ that is strictly contained in $K'_{j}$.  By construction, the matching conditions of
Lemma~\ref{productmatchinglemma} apply to the functions $u$ on $K$ and $u_{j}$ on $K'_{j}\times
K''$, so the piecewise defined function is smooth on the union of these cells.

Repeat the previous construction for each of the finite number of faces having the form
$\{q'_{i}\}\times K''$.  As these faces are disjoint we may choose the small cells in the
construction so that those used for $q'_{i}$ do not intersect those for $q'_{j}$ for $j\neq i$. The
result is a finite collection of cells $K'_{j}\times K''\subset U$ and functions $u_{j}$ such that
the piecewise function $u$ on $K$ and $u_{j}$ on $K'_{j}\times K''$ is smooth on the union of the
cells, and vanishes identically in a neighborhood of any boundary face of $\bigl(K'\cup(\cup
K'_{j})\bigr)\times K''$ that has the form $\{p'\}\times K''$.  We call this function $v'$.

Having treated the vertical faces $\{q'_{i}\}\times K''$, we then treat the horizontal faces
$\bigl(K'\cup(\cup K'_{j})\bigr)\times \{q''\}$ of the new function $v'$  in the same manner.  All
of the results we needed were valid on faces of finite unions of cells, so the same proof allows us
to piecewise extend to a smooth function $v$ on a larger finite union of cells, which we call $L$,
but with the additional condition that $v$ vanishes identically in a neighborhood of each
horizontal face of $L$.  Then $L\subset U$ and $v$ vanishes in a neighborhood of all faces of the
boundary of $L$, so Lemma~\ref{productmatchinglemma} ensures that extending $v$ to be identically
zero outside $L$ gives a smooth function on $X$.  By construction, $v=u$ on $K$.

For the estimate~\eqref{cutoffestimateonproduct} we note that
\begin{equation*}
    \bigl\|(\Delta'')^{m}\rho_{k}(x'')\bigr\|_{\infty}
    \leq \bigl\|(\Delta')^{k}(\Delta'')^{m}u \bigr\|_{L^{\infty}(K)}
    \end{equation*}
by definition, while rescaling~\eqref{estimatefornormalderivativesinproductusinghigherLap} to the
cell $K$ implies that
\begin{equation*}
    \bigl\|(\Delta'')^{m}\sigma_{k}(x'')\bigr\|_{\infty}
    \leq C \Bigl( (r'_{w'})^{-1} \bigl\|(\Delta')^{k}(\Delta'')^{m}u \bigr\|_{L^{\infty}(K)}
        + \mu'_{w'}\bigl\|(\Delta')^{k+1}(\Delta'')^{m}u \bigr\|_{L^{\infty}(K)}\Bigr).
    \end{equation*}
Substituting into~\eqref{Boreltheoremestimateforderivs} and using $r'_{w'(n_{l})}\leq r'_{w'}$ and
$\mu'_{w'}<1$ we have
\begin{align*}
    \bigl\|(\Delta')^{k}(\Delta'')^{m} \bigr\|_{\infty}
    &\leq \sum_{l=0}^{k} c(k,l) \bigl(r'_{w'(m_{l})}\mu'_{w'(m_{l})}\bigr)^{l-k}
        \bigl\|(\Delta')^{l}(\Delta'')^{m}u \bigr\|_{L^{\infty}(K)}\\
    &\leq C(k,U) \sum_{l=0}^{k} \bigl\|(\Delta')^{l}(\Delta'')^{m}u \bigr\|_{L^{\infty}(K)}.
    \end{align*}
This type of estimate deals with all of the vertical faces, and an analogous argument is valid for
the horizontal faces, so~\eqref{cutoffestimateonproduct} holds.
\end{proof}

\begin{theorem}\label{productparitioning}
If $u$ is smooth on $X$ and $\cup\Omega_{j}$ is an open cover of $X$ then there are constants
$C(k,m)$ and smooth functions $u_{j}$ such that $u_{j}$ is supported on $\Omega_{j}$,
$\sum_{j}u_{j}=u$, and
\begin{equation}\label{productpartitioningestimate}
    \| (\Delta')^{k}(\Delta'')^{m}u_{j} \|_{\infty}
    \leq C(k,m) \sum_{l=0}^{k} \sum_{n=0}^{m} \| (\Delta')^{l}(\Delta'')^{n}u \|_{\infty}.
    \end{equation}
\end{theorem}
\begin{proof}
The open cover is finite, say $\{\Omega_{j}\}_{1}^{J}$ because $X$ is compact.  Moreover we may
partition $X$ into a finite number of cells $K_{l}$ such that each $K_{l}$ is contained in some
$\Omega_{j}$. We proceed by induction on $l$, with the base case being that we apply
Theorem~\ref{cutoffonproductcell} to $u$ on $K_{1}$ to obtain a smooth function $v_{1}$ with
support in the open $\Omega_{j}$ that contains $K_{1}$.  At the $l$-th step we apply
Theorem~\ref{cutoffonproductcell} to $u-\sum_{1}^{l-1}v_{m}$ on $K_{l}$ to obtain a smooth function
$v_{l}$ with support in the open $\Omega_{j}$ that contains $K_{l}$. Note that
$u-\sum_{1}^{l}v_{m}$ vanishes on $\cup_{m=1}^{l} K_{m}$ so once we have exhausted the cells we
have $\sum_{l} v_{l}=u$.  By construction, each of the $v_{l}$ is smooth, supported on some
$\Omega_{j}$ and satisfies~\eqref{cutoffestimateonproduct}.  Setting $u_{j}$ to be the sum of those
$v_{l}$ that are supported on $\Omega_{j}$ completes the proof.
\end{proof}

\subsection*{Distributions with point support on products}

It is useful to begin with the observation that if $T'\in\distribson(X')$ and
$T''\in\distribson(X'')$ are distributions on the components of a product space $X'\times X''$ then
there is a tensor distribution $T'\times T''$ which is a distribution on the product.  This is not
entirely immediate, but follows readily from the structure theorem for the component spaces.
Specifically, the fact that $T'$ is locally $(-\Delta')^{k}f$ for a continuous $f$ implies that for
a $\phi\in\testfnson(X'\times X'')$ there are $k$ and $f$ such that
\begin{align*}
    \Delta'' T'\phi(x',x'')
    &= \Delta'' \int_{X'} f(x') (-\Delta')^{k}\phi(x',x'') \, d\mu'(x')\\
    &= \int_{X'} f(x') (-\Delta')^{k}\Delta''\phi(x',x'') \, d\mu'(x')\\
    &= T' \Delta''\phi(x',x'')
    \end{align*}
where we used that $\Delta'$ and $\Delta''$ commute.  In particular $T'\phi$ is smooth in the
second variable, so $T'\times T''\phi=T''(T'\phi)$ is well defined.  Repeating the calculation with
$T''$ in place of $\Delta''$ ensures that $T''(T'\phi)=T'(T''\phi)$, so the order in which the
distributions are applied is not important. Linearity of $T'\times T''$ is immediate and it is easy
to check the continuity condition that ensures it is a distribution on $X'\times X''$.

In the special case where $T'$ is supported at $x'$ and $T''$ is supported at $x''$ it is apparent
that $T'\times T''$ is supported at $(x',x'')$, so this construction and the results of
Section~\ref{distribsatpoint_section} supply a large number of distributions with point support. In
fact we can show that if $x'$ and $x''$ are either junction points or satisfy the conditions of
Theorem~\ref{distribsatapoint_identifydistribsatgenericpoints}, then the distributions with support
at $(x',x'')$ are of this type.  As in Section~\ref{distribsatpoint_section}, the key is to show
that if $\testfn\in\testfnson(X'\times X'')$ is annihilated by sufficiently large collection of
tensor distributions at $(x',x'')$ and if $\epsilon>0$ is given, then it is possible to cut off
$\testfn$ on a small neighborhood of $(x',x'')$ such that the the resulting function has
$\bigl\|(\Delta')^{j}(\Delta'')^{k}\testfn\bigr\|_{\infty}<\epsilon$ for all $j$ and $k$ such that
$j+k\leq m$.  It follows that all distributions of order at most $m$ and support $x'\times x''$ are
linear combinations of the given tensor distributions.

Our main tool is an adaptation of
Theorem~\ref{distribsatapoint_decayofHijtestfnimpliesgoodextension}.

\begin{theorem}\label{productcutoffwithdiameterestimatesonsmoothness}
Given a test function $\testfn$ and a cell $K=K'\times K''$ with $K'=F'_{w'}(X')$ and
$K''=F''_{w''}(X'')$, there is a test function $\psi$ such that $\psi=\testfn$ on $K$ and
\begin{equation}\label{products_estimateforpointsupportthm}
    \bigl\| (\Delta')^{j} (\Delta'')^{k} \psi \bigr\|_{\infty}
    \leq C(m,n) \sum_{l=0}^{m}\sum_{i=0}^{n} \bigl(r'_{w'} \mu'_{w'}\bigr)^{l-j} \bigl(r''_{w''} \mu''_{w''}\bigr)^{i-k}
    \bigl\|(\Delta')^{l}(\Delta'')^{i}\testfn\bigr\|_{L^{\infty}(K)} +\epsilon
    \end{equation}
for all $0\leq j\leq m$ and $0\leq k\leq n$.
\end{theorem}
\begin{proof}
The method for cutting-off a smooth function on a cell has already been described in the proof of
Theorem~\ref{cutoffonproductcell}.  Since we cut off first in one variable and then in the other,
the estimates from the proof of Theorem~\ref{distribsatapoint_decayofHijtestfnimpliesgoodextension}
may be applied directly. Suppose that we cut off in the first variable and then in the second.
Taking~\eqref{distribsatapoint_estimateofjetoffK} for the Laplacian $(\Delta')^{k}$ in the first
variable on a fixed slice $U'\times \{y''\}$ and substituting from the second lines of both
of~\eqref{distribsatapoint_testfnbyhijl} and~\eqref{distribsatapoint_normalderivoftestfnbyhijl},
gives
\begin{equation*}
    \bigl\| (\Delta')^{j} \psi \bigr\|_{L^{\infty}(U'\times \{ y''\})}
    \leq C(m) \sum_{l=0}^{m} \bigl(r'_{w'} \mu'_{w'}\bigr)^{l-j} \bigl\|(\Delta')^{l}\testfn\bigr\|_{L^{\infty}(K'\times
    \{y''\})} +\epsilon
    \end{equation*}
provided $j\leq m$.  In this calculation we used that the harmonic part of a function (which was
denoted $H_{i,j}$ in the proof of
Theorem~\ref{distribsatapoint_decayofHijtestfnimpliesgoodextension}) is bounded by the $L^{\infty}$
norm of the function because of the maximum principle, and we extracted the scaling factor $r'_{w'}
\mu'_{w'}$ of the Laplacian on $K'=F'_{w'}(X')$ using the same argument as
in~\eqref{distribsatapoint_comparablesizedcellstep}.

The same estimate is true with the same proof when $\psi$ is replaced by $(\Delta'')^{k}\psi$ and
$\testfn$ by $(\Delta'')^{k}\testfn$.  We use this fact when we repeat the estimate in the second
variable, because in this case we are cutting off the function that was modified at the first step.
A little algebra then produces the desired estimate.
\end{proof}

\begin{theorem}
Let $T$ be a distribution supported at $(x',x'')\in X'\times X''$.  Suppose that $x'$ is such that
either Theorem~\ref{distribsatapoint_identifyingdistribsatjunctionpoint} or
Theorem~\ref{distribsatapoint_identifydistribsatgenericpoints} may be used to identify the
distributions with support at $x'$, and make the same assumption for $x''$.  Then $T$ is a finite
linear combination of tensor products $T'\times T''$ where $T'$ is supported at $x'$  and $T''$ is
supported at $x''$.
\end{theorem}
\begin{proof}
In light of the preceding discussion and
Theorem~\ref{productcutoffwithdiameterestimatesonsmoothness}, it suffices to show that if the given
tensor distributions vanish on a test function $\testfn$ then the right side
of~\eqref{products_estimateforpointsupportthm} may be made less than $2\epsilon$ by taking $K$
sufficiently small.  The proof of this estimate is elementary: we simply go from $(x',x'')$ to
$(y',y'')$ by using two Taylor-like expansions, one in each variable.

Since $T$ has compact support it also has finite order $m$.  It then seems reasonable that each of
the terms $T'\times T''$ should be made up of a $T'$ of order $k\leq m$ and a $T''$ of order at
most $m-k$. Unfortunately we cannot prove this in general because our scaling estimates are
insufficiently refined, as was explained in Remark~\ref{distribs_commentaboutSGcase}. This result
is true if the distributions at $x'$ and $x''$ are such that none have scaling exactly equal to
that of the Laplacian (meaning that if they are as in Theorem~\ref{distribsatapoint_dsisadistrib}
then there is no $\gamma_{s}$ equal to a power of $r\mu$, and if they are as in
Corollary~\ref{distribsatapoint_genericordersofdistribs} then there is no $\gamma_{s}$ equal to a
power of $\beta_{w}$).  Given the limitations of our estimates we must instead allow the
possibility that $T'$ is order $k+1$ and $T''$ is order $m-k+1$.

Suppose then that  $T'\times T''\testfn=0$ for all $T'$ of order up to $k+1$ and $T''$ of order up
to $m-k+1$. It follows that the differential $(\bar{D}'')^{m-k}$ vanishes on the one-variable
smooth function $T'\testfn(x',\cdot)$. The same reasoning as was used at the beginning of the
proofs of Theorem~\ref{distribsatapoint_identifyingdistribsatjunctionpoint} and
Theorem~\ref{distribsatapoint_identifydistribsatgenericpoints} shows that then
$T'\testfn=o(r''_{w''}\mu''_{w''})^{m-k}$ on the set $\{x'\}\times K''$, so in particular at
$(x',y'')$.

We now wish to repeat the argument to go from $(x',y'')$ to $(y',y'')$.  Instead of having
vanishing distributions in the first variable at $(x',y'')$ we have only estimates on their size,
which we use to estimate the size of $(\bar{D}')^{m}\testfn$.   Recall from
Definitions~\ref{distribsatapoint_defnofdsatjunctionpt}
and~\ref{distribsatapoint_defnofdsatgenericpt} that the differential $(\bar{D}')^{m}\testfn$ for
the second variable on the cell $K'$ consists of a harmonic function with coefficients obtained
using distributions of order at most $m+1$, as well as $G'(\bar{D}')^{m-1}\Delta'\testfn$. where
$G'$ is the Green's operator for the cell $K'$.  The harmonic function is itself made up of pieces
(one for each $k\leq m$) with scaling bounded by $(r'_{w'}\mu'_{w'})^{k}$ (or an equivalent
quantity involving $\beta'_{w'}$) and coefficients obtained using distributions in the first
variable with order at most $k+1$. The estimate of the previous paragraph says that these
coefficients are $o(r''_{w''}\mu''_{w''})^{m-k}$, so each term of the harmonic functions is
$o\bigl((r'_{w'}\mu'_{w'})^{k}(r''_{w''}\mu''_{w''})^{m-k}\bigr)$ on $K'\times\{y''\}$.  A similar
argument applies to $G'(\bar{D}')^{m-1}\Delta'\testfn$, because the $G'$ produces an extra factor
of $r'_{w'}\mu'_{w'}$, and the harmonic piece of $(\bar{D}')^{m-1}\Delta'\testfn$ that has scaling
$(r'_{w'}\mu'_{w'})^{k-1}$ is obtained via distributions of order at most $k$ applied to
$\Delta'\testfn$, each of which is a distribution of order $k+1$ applied to $\testfn$.  We may
repeat this reasoning inductively across the terms of $(\bar{D}')^{m}\testfn$ to obtain a bound of
the form
\begin{equation}\label{products_taylorestimatestepone}
    \bigl| (\bar{D}')^{m}\testfn \bigr|
    =o \biggl( \sum_{k=0}^{m} (r'_{w'}\mu'_{w'})^{k}(r''_{w''}\mu''_{w''})^{m-k} \biggr).
    \end{equation}
on the set $K'\times\{y''\}$.  Since we also know (from~\eqref{differentialestimatetwo}
and~\eqref{distribsatapoint_genericbardiffestimate}) that
\begin{equation*}
    \bigl|\testfn- (\bar{D}')^{l}\testfn \bigr|
    = o (r'_{w'}\mu'_{w'})^{m}
    \end{equation*}
on $K'\times\{y''\}$ we conclude that the estimate~\eqref{products_taylorestimatestepone} is also
true for $\testfn$ itself.  The point $y''\in K''$ was arbitrary, so we have
\begin{equation}\label{product_taylorestfortestfn}
    \bigl\| \testfn \bigr\|_{L^{\infty}(K)}
    =o \biggl( \sum_{k=0}^{m} (r'_{w'}\mu'_{w'})^{k}(r''_{w''}\mu''_{w''})^{m-k} \biggr).
    \end{equation}

Our working thus far has shown that if $T'\times T''\testfn=0$ for all $T'$ of order up to $k+1$
and $T''$ of order up to $m-k+1$, then~\eqref{product_taylorestfortestfn} holds.  However, this
assumption obviously implies that $T'\times T''\Bigl((\Delta')^{l}(\Delta'')^{i}\testfn\Bigr)=0$
for all $T'$ of order up to $k+1-l$ and $T''$ of order up to $m-k+1-i$ if $l+i\leq m$ and $0\leq
k\leq (m-i-l)$. Thus~\eqref{product_taylorestfortestfn} improves to
\begin{equation}\label{product_taylorestfortestfnandderivs}
    \bigl\| (\Delta')^{l}(\Delta'')^{i} \testfn \bigr\|_{L^{\infty}(K)}
    =o \biggl( \sum_{k=0}^{m-i-l} (r'_{w'}\mu'_{w'})^{k}(r''_{w''}\mu''_{w''})^{m-i-l-k} \biggr).
    \end{equation}
Substituting into~\eqref{products_estimateforpointsupportthm} for the cutoff of $\testfn$ yields
\begin{equation*}
    \bigl\| (\Delta')^{j} (\Delta'')^{k} f \bigr\|_{\infty}
    \leq \epsilon + o\biggl( \sum_{l=0}^{m}\sum_{i=0}^{n}
    \bigl(r'_{w'} \mu'_{w'}\bigr)^{l-j} \bigl(r''_{w''}\mu''_{w''}\bigr)^{i-k}
    \sum_{s=0}^{m-i-l} (r'_{w'}\mu'_{w'})^{k}(r''_{w''}\mu''_{w''})^{m-i-l-k} \biggr).
    \end{equation*}
The simplest way to complete the argument is to choose $K'$ and $K''$ such that $r'_{w'}\mu'_{w'}$
and $r''_{w''}\mu''_{w''}$ are comparable, at which point all terms in the sum are bounded. It
follows that the sum term is $o(1)$ so can be made less than $\epsilon$ by requiring that $K'$ and
$K''$ are also sufficiently small.
\end{proof}

\section{Hypoellipticity}
An important question in the analysis of PDE is to identify conditions under which a distributional
solution of a PDE is actually a smooth function.  In Euclidean space, an archetypal example is
Weyl's proof that a weak solution of the Laplace equation is actually $C^{\infty}$.  In order to
study these questions one uses the notion of hypoellipticity, which we may now define in the
setting of fractafolds based on p.c.f. fractals and their products.  We will not settle any of the
questions about hypoellipticity here, but simply suggest some natural problems for which the
distribution theory we have introduced is the correct setting.

We first define the singular support of a distribution, which intuitively corresponds to those
points where the distribution is not locally smooth.

\begin{definition}
A distribution $T$ is smooth on the open set $\domain_{1}\subset\domain$ if there is
$u\in\smoothfnson(\domain_{1})$ such that
\begin{equation*}
    T\testfn = \int u\testfn\, d\mu \text{ for all $\testfn\in\testfnson(\domain_{1})$}
    \end{equation*}
\end{definition}

Using Lemma~\ref{testfunction_partition} for the case of a single p.c.f. fractal, or the analogous
result derived from Theorem~\ref{productparitioning} in the product setting, we see that if $T$ is
smooth on $\domain_{1}$ and on $\domain_{2}$ then it is smooth on $\domain_{1}\cup\domain_{2}$,
thus there is a maximal open set on which $T$ is smooth.

\begin{definition}
For a distribution $T$, Let $\domain_{T}$ be the maximal open set on which $T$ is smooth.  The
singular support of $T$ is the set
\begin{equation*}
    \singsppt(T)  = \sppt(T)\setminus\domain_{T}
    \end{equation*}
\end{definition}

Let $P$ be a polynomial of order $k$ on $\mathbb{R}^{m}$, so $P(\xi)=\sum_{|\kappa|\leq k}
a_{\kappa}\xi^{\kappa}$ where $\kappa=\kappa_{1}\dotsc\kappa_{m}$ is a multi-index, $|\kappa|=\sum
\kappa_{j}$ is its length, and $\xi^{\kappa}=\prod \xi_{j}^{\kappa_{j}}$. Consider the linear
differential operator $P(\Delta)=P(\Delta_{1},\dotsc,\Delta_{m})$ on a product $\prod X_{j}$ of
p.c.f. self-similar fractals $X_{j}$ with Laplacians $\Delta_{j}$. It is clear that for any
distribution $T$ we have $\singsppt\bigl(P(\Delta)T\bigr)\subseteq\singsppt(T)$, because when $T$
is represented by $u\in\smoothfnson(\domain_{T})$ then $P(\Delta)T$ is represented by $P(\Delta)u$.
By analogy with the Euclidean case, we define a class of constant coefficient hypoelliptic linear
differential operators.

\begin{definition}
$P(\Delta)$ is called hypoelliptic if $\singsppt\bigl( P(\Delta)T\bigr)=\singsppt(T)$ for all
$T\in\distribs$.
\end{definition}

Given the importance of hypoelliptic operators in the analysis of PDE on Euclidean spaces, it is
natural to seek conditions that imply hypoellipticity of an operator on a p.c.f. fractal or on
products of p.c.f. fractals. We expect that if $P(\Delta)$ is elliptic then it should  be
hypoelliptic; it also seems possible that the celebrated hypoellipticity criterion of H\"{o}rmander
\cite[Section~11.1]{MR705278} might imply hypoellipticity in the fractal case, though we do not
expect conditions of this type to be necessary because of examples like that motivating
Conjecture~\ref{quasiimplieshypoelliptic}.

\begin{definition}
For a polynomial $P(\xi)=\sum_{|\kappa|\leq k} a_{\kappa}\xi^{\kappa}$, the principal part of $P$
is $P_{0}=\sum_{|\kappa|=k} a_{\kappa}\xi^{\kappa}$. $P$ is called elliptic if $P_{0}(\xi)\neq 0$
for $\xi\neq0$; equivalently $P$ is elliptic if there is $c>0$ so $\bigl|P_{0}(\xi)\bigr|\geq
c|\xi|^{k}$.  We call $P(\Delta)$ elliptic if $P(\xi_{1}^{2},\dotsc,\xi_{m}^{2})$ is elliptic.
\end{definition}
\begin{remark}
The above definition is consistent with the usual one in the case that $X$ is a Euclidean interval
rather than a fractal set, but they do not coincide because we are dealing with a smaller class of
operators. Specifically, for such $X$ the Laplacian is $\partial^{2}/\partial x^{2}$, so our class
of operators $\{P(\Delta)\}$ is smaller than the usual collection of constant coefficient linear
partial differential operators $P(\partial/\partial x_{1},\dotsc,\partial/\partial x_{m})$.
Similarly our class of elliptic operators is a strict subset of the classical one.
\end{remark}

\begin{conjecture}
If $X$ is a p.c.f. fractal and $P(\Delta)$ is an elliptic operator on the product space $X^{m}$,
then $P(\Delta)$ is hypoelliptic.
\end{conjecture}

In the case that $m=1$, all operators $P(\Delta)$ are elliptic, and they can all be shown to be
hypoelliptic. Indeed, by factoring the polynomial we can reduce to the case of the linear
polynomial $\Delta+c$ for some complex constant $c$.  The hypoellipticity of $\Delta+c$ is readily
obtained from the fact that on small cells $(\Delta+c)$ has a resolvent kernel that is smooth away
from the diagonal, as may be seen by representing the resolvent as an integral with respect to the
heat kernel or by applying results from~\cite{IPRRS}.

\begin{conjecture}
A sufficient condition for the hypoellipticity of $P(\Delta)$ is that $D^{\kappa}P(\xi)/P(\xi)\to0$
as $\xi\to\infty$ for any partial derivative $D^{\kappa}$ with $|\kappa|>0$ (compare to
Theorems~11.1.1 and~11.1.3 of~\cite{MR705278}).
\end{conjecture}

In the Euclidean setting the above condition is necessary as well as sufficient, but we do not
expect this to be the case on fractals. In essence, the idea is that hypoellipticity of $P(\Delta)$
should depend only on whether the principal part $P_{0}(\Delta)$ is hypoelliptic, and that this is
equivalent (on the Fourier transform side) to estimates when inverting the algebraic equation
\begin{equation*}
    P_{0}(\lambda_{1},\dotsc,\lambda_{m})\hat{u}(\lambda_{1},\dotsc,\lambda_{m})=\hat{f}(\lambda_{1},\dotsc,\lambda_{m})
    \end{equation*}
for any choice of $(\lambda_{1},\dotsc,\lambda_{m})$ with each $\lambda_{j}$ an eigenvalue of
$\Delta_{j}$.  Since all of these $\lambda_{j}$ are negative, ellipticity of $P(\Delta)$ says that
$|P_{0}(\lambda_{1},\dotsc,\lambda_{m})|\geq c|\lambda_{1}+\dotsm+\lambda_{m}|$ for such
$(\lambda_{1},\dotsc,\lambda_{m})$, and this is sufficient to show the Fourier transform $\hat{u}$
has faster decay than $\hat{f}$, so $u$ should be as smooth or smoother than $f$.  However the
ellipticity condition should only be necessary if the points $(\lambda_{1},\dotsc,\lambda_{m})$ are
dense in the positive orthant $\{\xi:\xi_{j}\geq0\}$.  In~\cite{MR2333476} it is shown that this is
not the case for the Sierpinski Gasket fractal; specifically it is shown that in the case $m=2$,
the points $(\lambda_{1},\lambda_{2})$ omit an open neighborhood of a ray in the positive orthant.
It follows that there is $a>0$ and $b<0$ such that $a\Delta_{1}+b\Delta_{2}$ is not elliptic but
$-a/b$ lies in the omitted neighborhood, so $|a\lambda_{1}+b\lambda_{2}|\geq
c|\lambda_{1}+\lambda_{2}|$ whenever $\lambda_{j}$ is an eigenvalue of $\Delta_{j}$.
Following~\cite{MR2333476} we call operators of this type quasielliptic.  Given that quasielliptic
operators satisfy elliptic-type estimates on the spectrum, it seems likely that they will have
similar smoothness properties to elliptic operators; Lp estimates for these operators may be found
in recent work of Sikora~\cite{Sikora}.

\begin{definition}
The operator $P(\Delta)$ is quasielliptic if there is $c>0$ such that $|P_{0}(\xi)|\geq c|\xi|$ for
all $\xi\in\bigl\{(\lambda_{1},\dotsc,\lambda_{m}):\lambda_{j}\text{ is an eigenvalue of
}\Delta_{j}\bigr\}$.
\end{definition}

\begin{conjecture}\label{quasiimplieshypoelliptic}
The quasielliptic operators of~\cite{MR2333476} are hypoelliptic.
\end{conjecture}


\providecommand{\bysame}{\leavevmode\hbox to3em{\hrulefill}\thinspace}
\providecommand{\MR}{\relax\ifhmode\unskip\space\fi MR }
\providecommand{\MRhref}[2]{%
  \href{http://www.ams.org/mathscinet-getitem?mr=#1}{#2}
} \providecommand{\href}[2]{#2}

\end{document}